\documentclass[preprint,12pt]{elsarticle}

\usepackage[english]{babel}
\usepackage[latin1]{inputenc}
\usepackage{amsfonts}
\usepackage{graphicx}
\usepackage[all]{xy}
\usepackage{epsfig}
\usepackage{amsmath}
\usepackage{amsthm}
\usepackage{amssymb}

\DeclareMathOperator{\Hom}{Hom}

\DeclareMathOperator{\ind}{ind}
\DeclareMathOperator{\lnk}{lnk}
\DeclareMathOperator{\st}{st}

\newcommand{\N}{\mathbb{N}}
\newcommand{\Z}{\mathbb{Z}}
\newcommand{\R}{\mathbb{R}}

\newtheorem{theorem}{Theorem}[section]
\newtheorem{lemma}[theorem]{Lemma}

\newtheorem{cor}[theorem]{Corollary}

\newtheorem{prop}[theorem]{Proposition}

\newtheorem{question}[theorem]{Question}
\newtheorem*{maintheorem}{Main Theorem}

\begin{document}

\begin{frontmatter}



\title{Stiefel Manifolds and Coloring the Pentagon}


\author[ou]{James Dover\corref{cor1}\fnref{cu}}
\ead{jdover@cameron.edu}

\author[ou]{Murad \"{O}zayd\i n}
\ead{mozaydin@math.ou.edu}

\cortext[cor1]{Corresponding author}
\fntext[cu]{Present address: Department of Mathematical Sciences, Cameron University, Burch Hall, Room B001, 2800 W. Gore Blvd., Lawton, OK, USA 73505, Phone: 580-581-2488, Fax: 580-581-2616}

\address[ou]{Department of Mathematics, University of Oklahoma, 
601 Elm Avenue, Norman, OK, USA 73019-3103}

\begin{abstract}
We prove Csorba's conjecture that the Lov\'{a}sz complex
$\Hom(C_5,K_n)$ of graph multimorphisms from the $5$--cycle $C_5$
to the complete graph $K_n$ is $\Z/2\Z$--equivariantly
homeomorphic to the Stiefel manifold, $V_{n-1,2}$, the space of
(ordered) orthonormal $2$--frames in $\R^{n-1}$.  The equivariant piecewise-linear topology that we need is developed along the way.
\end{abstract}


\end{frontmatter}



\section{Introduction}
In his remarkable proof of the Kneser conjecture \cite{L78}, Lov\'{a}sz
gave a lower bound for the chromatic number of a graph using
equivariant algebraic topology.  This was essentially done via a functor
(the edge complex, defined below) from the category of graphs and
graph morphisms to the category $G$--$TOP$ (topological spaces equipped with an action of the group $G$
and $G$--maps, i.e., continuous functions commuting with the given actions).  In the case of the edge complex functor, the group $G$ is of order
two, and the actions are free. The edge complex of the complete
graph $K_n$ is (equivariantly homeomorphic to) the sphere
$S^{n-2}$ with the antipodal action, and the punchline is provided
by the Borsuk-Ulam theorem.

In \cite{C05}, Csorba observed that, for $n = 2$, $n = 3$, and $n = 4$, the $5$--cycle complex (defined below)
of the complete graph $K_n$, denoted  $\Hom(C_5,K_n)$,  is equivariantly homeomorphic to the Stiefel
manifold $V_{n-1,2}$ of (ordered) orthonormal 2-frames in the
Euclidean space $\R^{n-1}$, and he conjectured that this was true
for all $n$.  The non-equivariant
version of Csorba's conjecture was proven by C. Schultz in \cite{S08}, who also
proved that $\Hom(C_5,K_n)$ is equivariantly homotopy equivalent
to $V_{n-1,2}$.  In this note we give a proof of the
equivariant Csorba conjecture: $\Hom(C_5,K_n)$ is equivariantly
homeomorphic to $V_{n-1,2}$ with respect to the actions
described below.  This provides a rather mysterious combinatorial model for the Stiefel manifolds $V_{n-1,2}$ which are fundamental topological spaces.

The edge complex and the ($5$--)cycle complex are special cases of the Lov\'{a}sz graph multimorphism complex, which we proceed to define precisely now.  A \textbf{graph morphism} $f$ is a function from $V_{\Gamma}$ to
$V_{\Lambda}$ (where $V_{\Gamma}$ denotes the vertex set of the
graph $\Gamma$) so that the image of an edge of $\Gamma$ is an
edge of $\Lambda$.  (In particular, a graph morphism from $\Gamma$
to the complete graph $K_n$ is just an admissible coloring of the
vertices of $\Gamma$.)  A \textbf{multimorphism} from $\Gamma$ to $\Lambda$
is a relation $\phi \subseteq V_{\Gamma} \times V_{\Lambda}$ such that:
\begin{enumerate}
\item[(i)] $\phi(v) := \{w \in V_{\Lambda} | (v,w) \in \phi\}$ is non-empty for all $v \in V_{\Gamma}$, and
\item[(ii)] any function $f \subseteq \phi$ is a graph morphism.
\end{enumerate}
In other words, $\phi$ is a multimorphism when there is a function $f \subseteq \phi$ and any function $f \subseteq \phi$
is a graph morphism.

 The Lov\'{a}sz \textbf{multimorphism complex}
$\Hom(\Gamma, \Lambda)$ \cite{BK06} is a bifunctor (contravariant
in the first variable, covariant in the second) assigning a
regular cell complex to the (simple) graphs $\Gamma, \Lambda$. Any
regular cell complex is determined by its face poset
(\cite{BLSWZ99} p200, \cite{LW69} Ch 3 \S 1), and the face poset
of $\Hom(\Gamma, \Lambda)$ is the set of multimorphisms from
$\Gamma$ to $\Lambda$ ordered by inclusion. Geometrically, the
cells of the complex $\Hom(\Gamma, \Lambda)$ are products of simplices; the cell corresponding to the multimorphism $\phi$:
\begin{equation*}
\prod_{v\in V_{\Gamma}} \Delta \phi(v)
\end{equation*}
where $\Delta A$ is the full simplex on the vertex set $A$.
$\Hom(\Gamma, \Lambda)$ is the union of all these cells indexed by
the multimorphisms.  We will identify each cell with the
multimorphism $\phi$ indexing it and refer to it as such.  The
vertices ($0$--cells) of $\Hom(\Gamma, \Lambda)$ are the graph
morphisms from $\Gamma$ to $\Lambda$, the vertices of a given cell
$\phi$ are obtained by choosing any $w$ from $\phi(v)$ for each $v
\in V_{\Gamma}$.

The composition of two multimorphisms is also a multimorphism,
hence $\Hom(\Gamma, \Lambda)$ is (bi-)functorial for
multimorphisms (not just for morphisms, this extended
functoriality was put to very good use in \cite{S09-2}). That
is, a multimorphism $\alpha\colon\thinspace \Gamma^{\prime} \rightarrow \Gamma$
induces a (cellular) map $\Hom(\Gamma, \Lambda) \rightarrow
\Hom(\Gamma^{\prime}, \Lambda)$, and a multimorphism $\beta\colon\thinspace
\Lambda \rightarrow \Lambda^{\prime}$ induces a (cellular) map
$\Hom(\Gamma, \Lambda) \rightarrow \Hom(\Gamma,
\Lambda^{\prime})$.

If $G$ and $H$ are groups acting on the graphs $\Gamma$ and
$\Lambda$ respectively, then there is an action of $G \times H$ on
$V_{\Gamma} \times V_{\Lambda}$ inducing an action on the
relations, i.e., subsets of $V_{\Gamma} \times V_{\Lambda}$, which
restricts to an action on $\Hom(\Gamma, \Lambda)$.  In particular,
a graph $\Gamma$ equipped with an action of a group $G$ defines
the functor $\Hom(\Gamma, -)$ from the category of graphs and
graph (multi-)morphisms to $G$--$TOP$, the equivariant category of
$G$--spaces and $G$--maps.  Lov\'{a}sz's proof of the Kneser conjecture
\cite{L78} essentially employs the \textbf{edge complex} functor $\Hom(K_2,
-)$ with $G$ being the automorphism group of (the edge) $K_2$, the complete graph on $2$ vertices.

If we denote the vertices of $K_2$ with $+$ and $-$, then the
nontrivial element of $G$ is the involution on $\Hom(K_2,
\Lambda)$ switching $\phi(+)$ and $\phi(-)$, the subsets (of
$V_{\Lambda}$) assigned to the two vertices of $K_2$. This action
is free because $\phi(+)$ and $\phi(-)$ are distinct since they
are nonempty and disjoint. In particular, $\Hom(K_2,K_n)$ is a
poset consisting of pairs $(A, B)$ of nonempty disjoint subsets of $\{1, \ldots, n\}$, ordered by component-wise inclusion,
where $A = \phi(+)$ and $B = \phi(-)$.
$\Hom(K_2,K_n)$ can be geometrically realized as the $(n-2)$--sphere
of radius two with respect to the $L^1$ norm (inducing the taxicab
metric) in the hyperplane  orthogonal to the diagonal vector $(1, 1, \ldots, 1)$ in the Euclidean space $\R^n$, where $\phi(+)$ and
$\phi(-)$ are the sets of coordinates with positive and negative
values respectively. Switching $\phi(+)$ and $\phi(-)$ yields the
antipodal action.

Interestingly, long before the definition of $\Hom(\Gamma,
\Lambda)$, the underlying spaces of $\Hom(K_m, K_n)$ figured
prominently in two unrelated applications of equivariant algebraic
topology (neither one involving chromatic numbers or graphs): in
Alon's elegant Necklace Splitting Theorem (with $m$ prime)
\cite{A87} and the solution of the prime power case of the
Topological Tverberg Problem (\cite{O87} and \cite{Ziv98}), which was conjectured by
B\'{a}r\'{a}ny, Shlossman, and Sz\H{u}cs in \cite{BSS81}.

The Lov\'{a}sz conjecture (proven by Babson and Kozlov \cite{BK07},
see also \cite{S09-2}) is about the (odd) \textbf{cycle complex}es
$\Hom(C_m,\Lambda)$ where $C_m$ is the $m$--gon.
$\Hom(C_m,\Lambda)$ also has an involution induced by a reflection
(of the $m$--gon).  When $m$ is odd, any such reflection flips an
edge and hence induces a free action on $\Hom(C_m,\Lambda)$ as
above.  The Lov\'{a}sz conjecture reduces to a computation involving
the equivariant cohomology of $\Hom(C_m,K_n)$.

A graph multimorphism $\phi$ in $\Hom(\Gamma,\Lambda)$ is also determined
by specifying $\phi^{-1}(w) :=  \{v \in V_{\Gamma} | (v,w) \in
\phi \}$ for each $w \in V_{\Lambda}$.  Clearly, for any $\phi$,
each $\phi^{-1}(w)$  is independent (i.e., no two
elements are adjacent in $\Gamma$). The \textbf{independence complex}
$\ind(\Gamma)$ of a graph $\Gamma$ is the simplicial complex with
vertex set $V_{\Gamma}$ and simplices being independent subsets of $V_{\Gamma}$.  Equivalently,
$\ind(\Gamma)$ is the flag complex (the largest simplicial complex
whose $1$--skeleton is the given graph) of the edge complement of
$\Gamma$ (the graph with the same vertices as $\Gamma$ but with the complementary
edge set).  Any flag complex is
the independence complex of the edge complement of its
$1$--skeleton.

When $\Lambda = K_n$, the only condition on $\phi^{-1}(w)$ is
being in $\ind(\Gamma)$, so $\Hom(\Gamma,K_n)$ consists of $\phi$
such that: 
\begin{enumerate}
\item[(i)] $\phi(v)$ is non-empty for all $v \in V_{\Gamma}$, and 
\item[(ii)] $\phi^{-1}(j)$ is independent for $j = 1, \ldots, n$.
\end{enumerate}
We can identify $\Hom(\Gamma,K_n)_{>\phi}$ with the face poset of the
join (over $j = 1, \ldots, n$) of the links of the $\phi^{-1}(j)$
in $\ind(\Gamma)$. Since $\Hom(\Gamma,K_n)_{<\phi}$ is the face
poset of the boundary of a product (over the vertices $v \in
V_{\Gamma}$) of the simplices on $\phi(v)$, $\Hom(\Gamma,K_n)$ is
a (closed) manifold if $\ind(\Gamma)$ is a PL (piecewise-linear)
sphere.  In fact, the converse is also true \cite{C05}.  The
complex $\Hom(C_m,K_n)$ is a manifold only when $m = 5$ since this
is the only time when $\ind(C_m)$ is a sphere: $\ind(C_3)$ is 3
points, $\ind(C_4)$ is the disjoint union of two edges,
$\ind(C_5)$ is a pentagon, and $\ind(C_m)$ with $m > 5$ has maximal
simplices of different dimensions.

On the topological side, the orthogonal group $O_2$ acts on
the Stiefel manifold $V_{n-1,2}  := \{(x, y) \in S^{n-2} \times
S^{n-2} \; | \; x \cdot y = 0\}$ (with the Grassmannian $Gr_{n-1,2}$ of $2$--planes in $(n-1)$--space as the
quotient). The group $O_2$ is the semi-direct product of rotations
$SO_2$ with a subgroup generated by an arbitrary reflection. Two natural involutions on $V_{n-1,2}$ are (i) $(x,y)
\mapsto (x, -y)$ and (ii) $(x,y) \mapsto (y,x)$. Since any two
reflections are conjugate via a rotation, these give equivalent
actions.  An explicit map $V_{n-1,2} \rightarrow V_{n-1,2}$
interchanging the actions (i) and (ii) is $(x,y) \mapsto
\frac{1}{\sqrt{2}}(x+y, x-y)$.  Schultz \cite{S08} used the
action (i) on the Stiefel manifold. Below, we will use the
(equivalent) action (ii).  The corresponding involution on
$\Hom(C_5,K_n)$ is induced by any reflection of the pentagon $C_5$
(they all give equivalent actions).

It is convenient to work with a smaller model:
$\Hom_I(\Gamma,\Lambda)$, in which two multimorphisms are
considered the same if their values differ only on the
independent set $I$ of vertices.  Thus, $\Hom_I(\Gamma,\Lambda)$
is the subcomplex of $\Hom(\Gamma \setminus I,\Lambda)$ consisting
of the cells $\phi$ that can be extended to $\Gamma$.  The
projection from $\Hom(\Gamma,\Lambda)$ to $\Hom_I(\Gamma,\Lambda)$
is a homotopy equivalence \cite{C05} since the fibers are
contractible. It turns out that $\Hom(\Gamma,K_n)$ is homeomorphic (but not via the aforementioned projection)
to $\Hom_I(\Gamma,K_n)$ when $\ind(\Gamma)$ is a PL sphere
\cite{S08}. For $\Hom_I(\Gamma,\Lambda)$ to inherit the $G$--action
(induced from an action on $\Gamma$), the set $I$ needs to be
(setwise) $G$--invariant.  When $G$ is the group of order two
generated by the reflection on $C_5$ switching the vertex $i$ with
the vertex $6 - i$ for $i \in \{1, 2, 3, 4, 5\}$, there are two
$G$--invariant independent subsets: $\{3\}$ and $\{2, 4\}$.  In
\cite{S08} Schultz uses $I = \{2, 4\}$; we use $I = \{3\}$.  We
need $\Hom_I(C_5,K_n)$ and  $\Hom(C_5,K_n)$ to be equivariantly
homeomorphic which we get from an equivariant version of a lemma
of Schultz \cite{S08}.  

While  $\Hom(C_5,K_n)$ (or  $\Hom_{\{3\}}(C_5,K_n)$, which is
equivariantly homeomorphic to it) is the star of this story, the hero is
$\Hom_{\{3\}}(P_4,K_n)$, where $P_4$ is the path with four edges
from $1$ to $5$ (the subgraph of $C_5$ missing the edge $\{1,
5\}$). Our story also features $P_4 \setminus \{3\}$, which is the
disjoint union of two copies of $K_2$, namely the edges
$\{1, 2\}$ and $\{4, 5\}$.  The graph
morphism $P_4 \setminus \{3\} \rightarrow K_2$ sending $1$ and $5$
to $+$ and $2$ and $4$ to $-$ induces the diagonal embedding of
(the sphere) $\Hom(K_2,K_n)$ in $\Hom(P_4 \setminus \{3\},K_n) =
\Hom(K_2,K_n) \times \Hom(K_2,K_n)$.  We also have the restriction
map $\Hom_{\{3\}}(C_5,K_n) \rightarrow \Hom_{\{3\}}(P_4,K_n)$
induced from the inclusion of $P_4$ in $C_5$ and the inclusion of
$\Hom_{\{3\}}(P_4,K_n)$ into  $\Hom(P_4 \setminus \{3\},K_n)$. All
these maps are equivariant with respect to the involutions induced
from the aforementioned $i \mapsto 6 - i$.  (The action on $\Hom(K_2,K_n)$ is trivial since it maps homeomorphically to
the diagonal in $\Hom(P_4 \setminus \{3\},K_n)$, the fixed point
set of the involution.)

Note that $\Hom_{\{3\}}(P_4,K_n)$ is not homeomorphic to
$\Hom(P_4,K_n)$: We will show that $\Hom_{\{3\}}(P_4,K_n)$ is a
$(2n - 4)$--dimensional manifold with boundary while
$\Hom(P_4,K_n)$ is not a manifold, having maximal cells of every
dimension from $2n-4$ to $3n - 6$.  (Recall that $n \geq 3$.)

The face posets of these Lov\'{a}sz complexes above are actually
very easy to describe concretely; all are made up of (ordered) quadruples
of nonempty subsets $A, B, C, D$ of $\{1, \ldots , n\}$ satisfying
further conditions.  For any cell $\phi$, $A, B, C, D$ are
$\phi(1), \phi(2), \phi(5)$, and $\phi(4)$ respectively.  In
$\Hom(P_4 \setminus \{3\}, K_n)$ we have $A \cap B = \emptyset = C
\cap D$.  In $\Hom_{\{3\}}(P_4,K_n)$ we also have that $B \cup D
\neq \{1,\ldots,n\}$, and $\Hom_{\{3\}}(C_5,K_n)$ has the
further restriction that $A \cap C = \emptyset$.  The diagonal
$\Hom(K_2,K_n)$ has $A = C$ and $B = D$ (in addition to $A \cap B
= \emptyset$).  We will denote these cells as
\begin{equation*}
\phi = \left(
   \begin{array}{ccc}
    A & B \\
    C & D \\
   \end{array}\right)
\end{equation*} to remind us of their locations in terms of the vertices of the pentagon $C_5$.
Such an array should not be thought of as a $2 \times 2$ matrix.

On the topological side we have the involution on $S^{n-2} \times
S^{n-2}$ switching the two coordinates, with the diagonal
$S^{n-2}$ as the fixed set.  An equivariant regular neighborhood
of the diagonal is $N := \{(x, y) \in S^{n-2} \times S^{n-2} \; |
\; x \cdot y \geq 0\}$, which is a manifold with boundary
$V_{n-1,2} = \{(x, y) \in S^{n-2} \times S^{n-2} \; | \; x \cdot y
= 0\}$.   We prove in Section 5 that $\Hom_{\{3\}}(P_4,K_n)$ is a
manifold with boundary $\Hom_{\{3\}}(C_5,K_n)$ and also an
equivariant regular neighborhood of the diagonal in $\Hom(P_4
\setminus \{3\},K_n) \cong \Hom(K_2,K_n) \times \Hom(K_2,K_n)$. From the equivariant PL theory of
collapsing, shelling, and regular neighborhoods developed in
Section 4, our main result follows:

\begin{maintheorem}
The regular cell complex $\Hom_{\{3\}}(P_4,K_n)$ is a PL manifold with boundary
$\Hom_{\{3\}}(C_5,K_n)$ and is equivariantly homeomorphic (with respect to the involution described above) to $N :=
\{(x,y) \in S^{n-2} \times S^{n-2} \; | \; x \cdot y \geq 0 \}$,
where the involution on $N$ interchanges $(x,y)$ with $(y,x)$.
The Stiefel manifold $V_{n-1,2} = \partial N$ is therefore equivariantly homeomorphic to $\Hom_{\{3\}}(C_5,K_n)$, which is equivariantly homeomorphic to $\Hom(C_5,K_n)$.
\end{maintheorem}

The structure of the rest of the paper is as follows: We establish
notation and give basic definitions and prove that $\Hom(C_5,K_n)$
and $\Hom_{\{3\}}(C_5,K_n)$ are equivariantly homeomorphic in
Section 2.  The small amount of discrete Morse theory that we use
is in Section 3.  Section 4 is devoted to the stating the equivariant versions
of facts we need from piecewise-linear topology given in \cite{RS82}.  Finally,
in Section 5 we specialize to our setup and prove the equivariant
Csorba conjecture.

\section{Notation and Basics}
All simplicial complexes and posets we consider are finite. An
(abstract) \textbf{simplicial complex} $K$ on a vertex set $V$ is
a collection of (finite) subsets of $V$ such that if $\sigma
\subseteq \tau \in K$, then $\sigma \in K$.  The dimension of the
simplex $\sigma$ is one less than its cardinality, and the
\textbf{$n$--skeleton} of the simplicial complex $K$ (denoted
$K^n$) is the sub-simplicial complex made up of all simplices of
$K$ with dimension at most $n$. However, we will abuse the
notation and use $K^0$ to also mean $\bigcup \, \{\sigma \, | \,
\sigma \in K\}$. The nonempty simplices $\sigma \in K$ are also
called the faces of $K$.  The \textbf{geometric realization} of
$K$ is the topological space
\begin{equation*}
|K| := \{\; \sum_{v \in \sigma} t_v \delta_v \in \R^{V} \; | \; \sum_{v \in \sigma} t_v = 1, \, t_v > 0,\, \sigma \in K \setminus \{ \emptyset\}\}
\end{equation*}
where $\delta_v$ is the standard basis vector of $\R^V$
corresponding to $v \in V$.  The underlying topological space of a face $\sigma$ of $K$ is
\begin{equation*}
|\sigma| := \{\; \sum_{v \in \sigma} t_v \delta_v \in \R^{V} \; |
\; \sum_{v \in \sigma} t_v = 1, \, t_v \geq 0\}.
\end{equation*}
Note that we are using vertical bars to denote both the
cardinality of a finite set and the geometric realization of a
simplex or a simplicial complex; which one is meant should be
clear from the context.

If $\sigma$ is a simplex in a simplicial complex $K$, its \textbf{link} is defined $\lnk_K(\sigma) : = \{\tau \in K \, | \, \sigma \cap \tau = \emptyset, \sigma \cup \tau \in K\}$.  If $K$ and $L$ are two simplicial complexes, we define their \textbf{join} $K * L : = \{\sigma \cup \tau \, | \, \sigma \in K, \tau \in L\}$.  

A subcomplex $L$ of a simplicial complex $K$ is called \textbf{full} if it satisfies the property that if $\sigma \in K$ and $\sigma \subseteq L^0$, then $\sigma \in L$.

A \textbf{regular cell structure} on a (compact Hausdorff)
topological space $X$ is a (finite) collection $\{ c\}$ of subspaces
(called cells or faces), each homeomorphic to a (closed) disk of
dimension $d$ for some $d$, such that (1) $X$ is the disjoint
union of the relative interiors of its cells (called the open
cells and denoted by int $c$), and (2) the boundary of each cell is a union of (lower
dimensional) cells. A topological space with a regular cell
structure is a \textbf{regular cell complex}.  For any simplicial complex
$K$, the collection $\{|\sigma|\}_{\sigma \in K \setminus \{\emptyset \}}$ gives a regular
cell structure on $|K|$.

For any poset $P$, its \textbf{order complex} $\Delta P$ is the simplicial complex whose simplices are chains $a_0 < a_1 < \ldots < a_d$, with $a_i \in P$.  The faces of a regular cell complex under inclusion form a poset whose order complex is its barycentric subdivision (\cite{BLSWZ99} p200, \cite{LW69} Ch 3 \S 1).  Note that this notation is consistent with $\Delta A$ being the full simplex on a vertex set $A$ if the set $A$ is endowed with any total order.  Also note that an abstract simplicial complex $K$ is itself a poset and includes the empty simplex (unlike its face poset), so the order complex $\Delta K$ is the cone of the barycentric subdivision of $K$.

A \textbf{$G$--simplicial complex}, where $G$ is a (finite) group, is a simplicial complex $K$ equipped with a permutation action of $G$ on the vertex set $V$ so that the induced action on the subsets of $V$ sends simplices to simplices.  Similarly, in a $G$--cell complex, the group action permutes the cells.  For any cell (or simplex) $c$, $G_c := \{g \in G \, | \, gc = c\}$ is called its \textbf{stabilizer}. It will be convenient for us to define a \textbf{$G$--regular cell complex} to be a topological space with a $G$--action and a regular cell structure with the group $G$ permuting the cells so that every closed cell is $G_c$--homeomorphic to a cone on its boundary with the apex fixed by $G_c$.  The geometric realization of a $G$--simplicial complex $K$ is a $G$--regular cell complex with cells $\{|\sigma|\}_{\sigma \in K \setminus \{\emptyset \}}$ since the stabilizer of a simplex always fixes its centroid.  Similarly, if the cells of a regular cell complex are convex with $G$ acting affinely on each one, the complex is $G$--regular. If the stabilizer of each cell (or simplex) fixes the cell pointwise, the complex is called \textbf{admissible}.  If $G$ acts on a poset $P$ (preserving the order), then $\Delta P$ is an admissible $G$--simplicial complex.

Just as in the non-equivariant case (\cite{BLSWZ99} p200,
\cite{LW69} Ch 3 \S 1), the face poset of a $G$--regular cell
complex determines its $G$--homeomorphism type:

\begin{lemma}\label{Gcellcx}
If X is a G-regular cell complex with face poset F, then $X$ is $G$--homeomorphic to $|\Delta F|$.
\end{lemma}

The proof below is very similar to the non-equivariant case and is
mainly given for completeness.

\begin{proof}
The proof is by induction on the number of cells in $X$.  There is nothing to prove if X consists of a single $0$--cell.  Now, choose a maximal cell $c$ and define $Y := X \setminus \bigcup_{g \in G} \text{int } gc$ with the induced $G$--cell structure.  By the induction hypothesis, $Y \approx_G |\Delta (F \setminus Gc)|$.  Also, $\partial c \approx_{G_c} |\Delta F_{< c}|$.  Then we take the cone of $\partial c$ with its apex being a point $x \in \text{int } c$ fixed by $G_c$.  Extending this coning equivariantly to $G \partial c$ gives the homeomorphism from $|\Delta F|$ to $X$.
\end{proof}

We are specifically interested in the $G$--equivariant homeomorphism type of some Lov\'{a}sz multimorphism complexes $\Hom(\Gamma,K_n)$, for a graph $\Gamma$ with a group of symmetries $G$.  The following equivariant version of Lemma 3.5 in \cite{S08} enables us, in some cases, to work with the smaller and more convenient $G$--regular cell complexes $\Hom_I(\Gamma, K_n)$, where $I$ is a $G$--invariant independent subset of vertices of $\Gamma$.

\begin{lemma}\label{schultz2}
Let $\Gamma$ be a graph with a G-action, n $\geq$ 1, and $I$ a
$G$--invariant, independent subset of the vertex set $V_{\Gamma}$.  For all $v \in V_{\Gamma}$,
define

\begin{equation*}
A_v := \{J \in \ind(\Gamma) \, | \, v \in J \}, \quad B_v := \{ J \setminus I \, | \, J \in A_v \}
\end{equation*}

If there is a $G$--homeomorphism $h\colon\thinspace |\Delta \ind(\Gamma)|
\rightarrow |\Delta \ind(\Gamma \setminus I)|$ such that
$h(|\Delta A_v|) = |\Delta B_v|$ for all $v \in V_{\Gamma}$, then
$\Hom(\Gamma,K_n)$ is $G$--homeomorphic to $\Hom_I(\Gamma,K_n)$.
\end{lemma}

\begin{proof}Following \cite{S08}, we
consider the equivariant poset embedding
\begin{equation*}
f: \Hom(\Gamma,K_n) \rightarrow \prod_{i = 1}^n \;
\text{ind}(\Gamma) = \ind(\Gamma)^{\{1, \ldots, n\}}
\end{equation*}
given by $\phi \mapsto (\phi^{-1}(i))_i$.  Then $\Hom(\Gamma,K_n) \approx_G
|\Delta \text{im } f|$. The poset $\ind(\Gamma)^{\{1, \ldots, n\}}$ can naturally be identified with those
relations $\phi \subseteq V_{\Gamma} \times \{1, \ldots, n\}$ that are multimorphisms from the induced subgraph on the vertices with  $\phi(v)$ nonempty to the complete graph $K_n$.  The
additional condition that no $\phi(v)$ can be empty yields the
following description:
\begin{equation*}
\text{im } f = \bigcap_{v \in V_{\Gamma}} \bigcup_{j = 1}^n
\prod_{i =1}^n \left\{\begin{array}{lr}
A_v, & i = j \\
\text{ind}(\Gamma), & i \neq j
\end{array}\right.
\end{equation*}
All the $A_v$ satisfy the condition that if $x \in A_v$ and $x
\leq y$, then $y \in A_v$.  Therefore, taking the order complex
commutes with unions, and we obtain that
\begin{equation*}
\Hom(\Gamma,K_n) \approx_G \bigcap_{v \in V_{\Gamma}} \bigcup_{j =1}^n \prod_{i = 1}^n \left\{\begin{array}{lr}
|\Delta A_v|, & i = j \\
|\Delta \ind(\Gamma)|, & i \neq j
\end{array}\right.
\end{equation*}
We use a similar argument for $\Hom_I(\Gamma,K_n)$.  The image of its
embedding in $\ind(\Gamma \setminus I)^{\{1,\ldots,n\}}$ has the additional
condition that for each vertex in $I$, there is some element of
$\{1,\ldots,n\}$ that is not related to any of its neighbors in $\Gamma$.  We
have that, for $v \notin I$, $B_v$ satisfies the same condition as
$A_v$ above.  For $v \in I$, $B_v$ also satisfies the condition
that if $x \in B_v$ and $y \leq x$, then $y \in B_v$.  Hence,
\begin{equation*}
\Hom_I(\Gamma,K_n) \approx_G \bigcap_{v \in V_{\Gamma}} \bigcup_{j = 1}^n \prod_{i =1}^n \left\{\begin{array}{lr}
|\Delta B_v|, & i = j \\
|\Delta \ind(\Gamma \setminus I)|, & i \neq j
\end{array}\right.
\end{equation*}
Thus, using the $G$--homeomorphism from the hypothesis on each
coordinate in the product, we obtain that $\Hom(\Gamma,K_n)
\approx_G \Hom_I(\Gamma,K_n)$.
\end{proof}

In our case, we choose $I = \{3\}$ as our independent set in $C_5$.
The existence of a $G=\{\pm 1\}$--homeomorphism $h\colon\thinspace |\Delta \ind(C_5)|
\rightarrow |\Delta \ind(C_5 \setminus \{3\})|$ satisfying
the necessary conditions is easy to see from Figure \ref{fig1}. Thus, we have $\Hom(C_5,K_n)$ is equivariantly homeomorphic to $
\Hom_{\{3\}}(C_5,K_n)$.

\begin{figure}
\includegraphics[scale=.4]{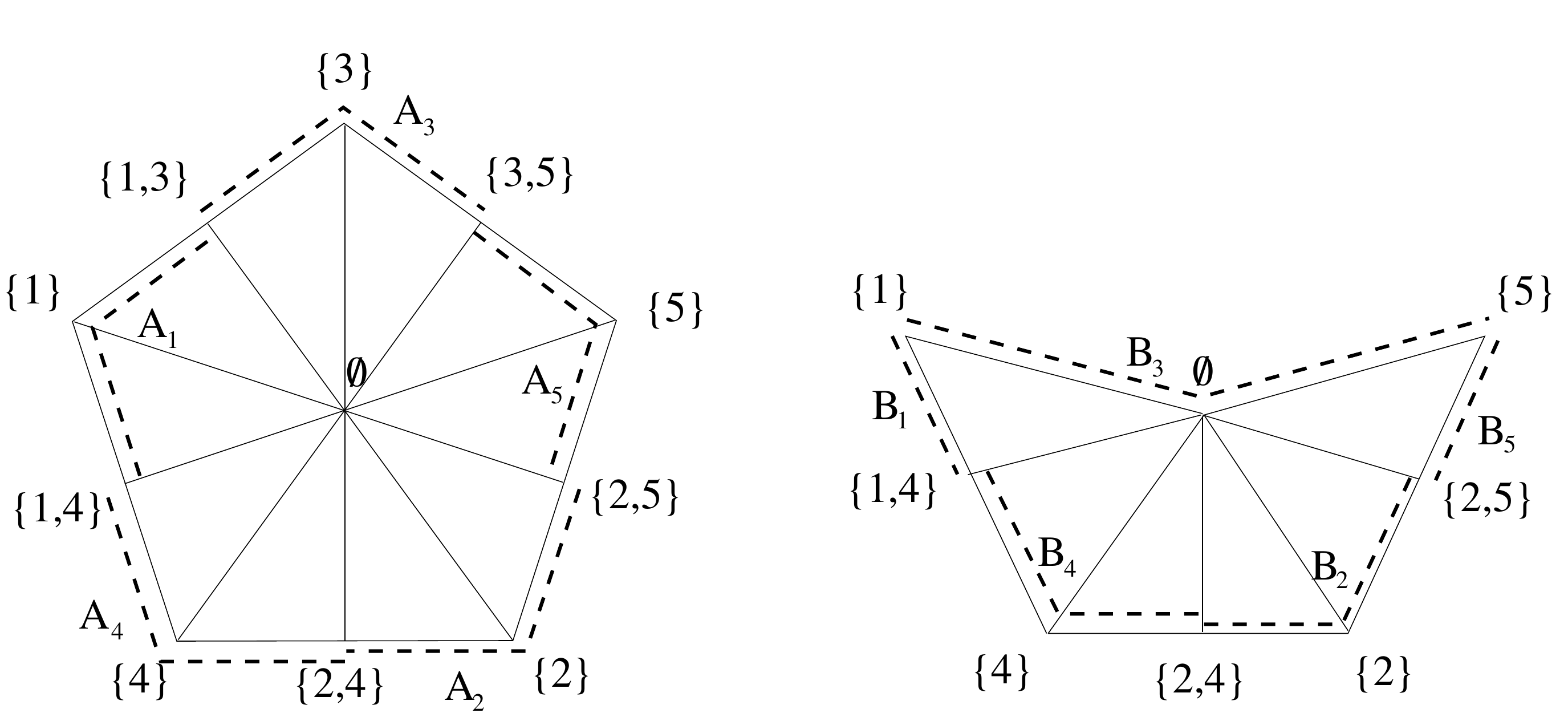}
\caption{$|\Delta \ind(C_5)|$ and $|\Delta \ind(C_5
\setminus \{3\})|$ \label{fig1}}
\end{figure}

\section{Discrete Morse Theory}
We will use Robin Forman's Discrete Morse Theory to find
collapsings of simplicial complexes.  More thorough discussions of
the subject can be found in \cite{F98} and \cite{K07}.

Let $K$ be a finite (abstract) simplicial complex.  We use the notation
$\sigma \lessdot \tau$ if $\sigma < \tau$ and dim $\sigma = \text{dim}
\, \tau - 1$.  By a vector, we mean a pair $(\sigma \lessdot
\tau)$, where $\tau$ is thought of as the head of the vector and
$\sigma$ the tail.  A \textbf{discrete vector field} on $K$ is defined to be a
collection of vectors $V = \{\, (\sigma_i \lessdot \tau_i) \, |
\, i \in I \}$ such that each simplex $\rho \in K$ belongs to
at most one element of $V$, either as a head or a tail of a
vector.

Given a discrete vector field $V$ on $K$, we have the notion of a
\textbf{path}, which is a sequence of simplices in K of the form:
\begin{equation*}
\sigma_0 \lessdot \tau_0 \gtrdot \sigma_1 \lessdot \tau_1 \gtrdot
\ldots \gtrdot \sigma_{s-1} \lessdot \tau_{s-1} \gtrdot \sigma_{s}
\end{equation*}
where $\forall \; i: \; 0\leq i < s$, $(\sigma_i \lessdot \tau_i)
\in V$ and $\sigma_i \neq \sigma_{i+1}$.  We say a path as above has length $s$.
By a \textbf{cycle} we mean a path as above with $\sigma_{s} =
\sigma_0$.

A \textbf{Morse matching} (or a \textbf{discrete gradient field}) is
a discrete vector field $V$ with no cycles.  The simplices which are
unpaired in $V$ are called \textbf{critical}.

An equivalent concept to a Morse matching is a height function on
$K$.  A \textbf{height} (or \textbf{Morse})  \textbf{function} is a map $h\colon\thinspace K
\rightarrow \R$ satisfying $\forall \; \sigma \in K$,
\begin{equation*}
|\{\rho \lessdot \sigma \, | \, h(\rho) \geq h(\sigma)\}
\cup \{\tau \gtrdot \sigma \, | \, h(\tau) \leq h(\sigma)\}|
\leq 1
\end{equation*}
Given a height function $h$, the corresponding Morse matching is the collection of pairs $(\sigma \lessdot \tau)$ for which $h(\sigma) \geq h(\tau)$.  Conversely, it is not difficult to construct a height function inducing a given Morse matching (\cite{F98}, \cite{K07}).  This height function is clearly not unique.  In fact, we may adjust it to be one-to-one and to take values in $\N$ (without changing the Morse matching).

The following is the key lemma from discrete Morse theory we will use in this paper.  A more general version for cell complexes can be found in \cite{K07}, Theorem 11.13.

\begin{lemma}\label{morse2}
Let K be a finite simplicial complex with a Morse matching
whose critical simplices form a subcomplex L.  Then K collapses
simplicially to L.
\end{lemma}
\begin{proof} Let $h\colon\thinspace K \rightarrow \N$ be a one-to-one height
function corresponding to the given Morse matching.  Define a new
height function $\tilde{h}\colon\thinspace K\rightarrow \N$ as follows:
For $\sigma \in L$, set $\tilde{h}(\sigma) = \text{dim} \,(\sigma)$. For
$\sigma \notin L$, set $\tilde{h}(\sigma) = h(\sigma) + \text{dim}(L)$.
Under this new height function, all of the simplices in $L$ remain
critical, and the relative heights of all the simplices outside of
$L$ are unchanged, preserving their pairings. Therefore,
$\tilde{h}$ corresponds to the same Morse matching as $h$, and
$\tilde{h}$ is one-to-one on $K \setminus L$.

Now, for $m \in \N$ define
\begin{equation*}
K(m):= \{\sigma \in K\; | \; \exists \; \tau \geq \sigma \text{
such that } \tilde{h}(\tau) \leq m\}
\end{equation*}
Note that $K(\text{dim}(L)) = L$. For $m \geq \text{dim}(L)$, either $K(m + 1) =
K(m)$ or $K(m + 1) = K(m) \cup \{\rho, \tau\}$ where $\rho
\lessdot \tau$ and $\tilde{h}(\tau) = m + 1 < \tilde{h}(\rho)$. In
the latter case, $K(m + 1)$ collapses to $K(m)$ along the free
face $\rho$. Since $K = K(\text{max}\{\tilde{h}(\sigma) \, | \,
\sigma \in K\})$, $K$ collapses to $L$ via a sequence of these
collapsings.
\end{proof}

\section{Equivariant Neighborhoods}
In this section, we discuss the equivariant theory of $G$--regular neighborhoods with suitable hypotheses, similar to Rourke and Sanderson's non-equivariant version \cite{RS82}.   Throughout, we assume $G$ is a finite group, all simplicial complexes are finite, and maps between polyhedra are piecewise-linear. Also, when there is a $G$--action on a space $X$, the product $X \times I$ is assumed to have the $G$--action $g(x,t) = (gx, t)$.

A simplicial complex $K$ \textbf{polyhedrally triangulates} a
space $X \subset \R^n$ if the $0$--skeleton $K^0$ is identified
with a subset of $X$ so that the canonical piecewise-linear
function from $|K|$ to $\R^n$ mapping $|\sigma|$ of each simplex
$\sigma$ to the convex hull of its vertices in $X$ is a
homeomorphism onto $X$.  Then $X$ is called a \textbf{polyhedron},
and two polyhedra are equivalent if there is a piecewise-linear
homeomorphism between them. If $K$ is a $G$--complex, the ambient
vector space $\R^n$ into which the polyhedron $X$ is $G$--embedded
is the underlying space of a (real) linear representation of the
group $G$, and the piecewise-linear homeomorphism from $|K|$ to
$X$ is a $G$--map, then we have a \textbf{polyhedral
$G$--triangulation} of a \textbf{$G$--polyhedron} $X$. We implicitly
identify each simplex $\sigma$ of $K$ with the image of its
geometric realization in $X$. Note that one abstract
($G$--)simplicial complex may triangulate $X$ in multiple ways by
having different choices for some of the vertices in $X$, in which
case we will give the simplicial complex different names. We will
only consider polyhedral triangulations.

Let $Y \subset X$ be polyhedra and $L$ be a subcomplex of $K$ such
that $(K, L)$ triangulates $(N, Y)$ where $N$ is a neighborhood of
$Y$ in $X$. Define the \textbf{simplicial neighborhood} of $L$ in
$K$, $N_K(L) := \{\, \sigma \in K \, | \, \exists \, \tau \in K
\text{ s.t. } \sigma < \tau \text{ and } \tau \cap L^0 \neq
\emptyset \}$. Also, define $\dot{N}_K(L) := \{\, \sigma \in
N_K(L) \, | \, \sigma \cap L^0 = \emptyset \}$. The
\textbf{derived subdivision} $K^{\prime}$ of $K$ near $L$ is
(combinatorially) the simplicial complex with vertex set $K^0 \cup
\{v_{\tau} \, | \, \tau \in K \setminus L, \, \tau \cap L^0 \neq
\emptyset \}$, and the simplices are of the form $\sigma \cup
\{v_{\tau_1}, \ldots v_{\tau_m}\}$ where $\sigma \in L$ or $\sigma
\in K$ with $\sigma\cap L^0 =\emptyset$ and $\sigma < \tau_1 <
\ldots < \tau_m$. Geometrically, $K^{\prime}$ is realized by
selecting the new vertex $v_{\tau}$ in the interior of each
simplex $\tau$ of $K \setminus L$ that intersects a simplex of $L$
and then, in ascending order of dimension, replacing each $\tau$
with the cone (with apex $v_{\tau}$) on its boundary (which has
already been subdivided in the previous steps).

Now suppose $L$ is full in $K$ (i.e., if a set of vertices in $L$
forms a simplex in $K$, then they form a simplex in $L$) and
$|\dot{N}_K(L)| =
\partial|N_K(L)|$ in $X$. Then $N_1 = |N_{K^{\prime}}(L)|$ is called a
\textbf{regular neighborhood} of $Y$ in $X$. If $K$ and $L$ are
both admissible $G$--complexes and, when defining $K^{\prime}$, the
set of new vertices $\{v_{\tau}\}$ is chosen to be $G$--invariant,
we say $N_1$ is a \textbf{$G$--regular neighborhood}.

\begin{prop}\label{reg1}
If K is a $G$--simplicial complex, $|K| \times [0,1]$ has an admissible $G$--triangulation with no new vertices in $|K| \times (0,1)$.
\end{prop}
\begin{proof} Let $F$ be the face poset of $K$.
Then the order complex $\Delta(F \times \{0,1\})$ of the product
poset $F \times \{0,1\}$ satisfies $|\Delta (F \times \{0, 1\})|
\approx_G |\Delta F| \times |\Delta \{0,1\}| \approx_G |K| \times
I$. Hence $\Delta(F \times \{0,1\})$ gives the desired
$G$--triangulation.
\end{proof}

Let $\mathcal{K} = \{K_1, K_2, \ldots, K_s\}$ be a collection of
simplicial complexes, each one triangulating the polyhedron $X$.
For $\phi \in K_1 \times \ldots \times K_s$, let $|\phi| :=
\bigcap_{1 \leq i \leq s} |\phi_i|$.  Define an equivalence
relation on $K_1 \times \ldots \times K_s$ by $\phi \sim \psi
\Leftrightarrow |\phi| = |\psi|$. We define  a poset
$C_{\mathcal{K}} := \{\phi \in K_1 \times \ldots \times K_s \, |
\, |\phi| \neq \emptyset\}/\sim$ with the partial order $[\phi]
\leq [\psi] \Leftrightarrow |\phi| \subseteq |\psi|$. We will
identify an equivalence class $[\phi]$ with the geometric
realization $|\phi|$ of its representatives.

\begin{prop}\label{commonsubdivision}
Given a collection $\mathcal{K} = \{K_1, \ldots, K_s\}$ of
simplicial complexes that triangulate a polyhedron X, then
$C_{\mathcal{K}}$ is a regular cell structure on X, so that
$|\Delta C_{\mathcal{K}}|$ is a triangulation of X and a common
subdivision of $K_1, \ldots, K_s$.
\end{prop}
\begin{proof}
We show first that the open cells of $C_{\mathcal{K}}$ are disjoint.  In this discussion, let $|\phi| = \bigcap_{1 \leq i \leq s} |\sigma_i|$ and $|\psi| = \bigcap_{1 \leq i \leq s} |\tau_i|$.  Note that if $|\phi| \cap |\psi| \neq \emptyset$, we have that $|\phi| \cap |\psi| = \bigcap_{1 \leq i \leq s} (|\sigma_i| \cap |\tau_i|) =
\bigcap_{1 \leq i \leq s} |\sigma_i \cap \tau_i|$.  Thus, the intersection of any two closed cells is a closed cell.
If $[\phi] < [\psi]$, then, without changing $[\psi]$ or $[\phi]$, we may replace each $\tau_i$ with its minimal face containing $|\psi|$ and each $\sigma_i$ with $\sigma_i\cap\tau_i$.  Then $|\psi| \cap \text{int}(|\tau_i|)$ is always nonempty, and  for some $i$ we have
$|\phi| \subseteq |\partial \tau_i|$.  Therefore we have that $|\phi| \subseteq \partial |\psi|$, implying that the intersection of any two distinct closed cells must occur on the boundary of at least one of them, so any two distinct open cells are disjoint.

Each $|\psi|$ is a (nonempty) compact, convex polytope, yielding
that $|\psi| \approx D^m$ for some $m \geq 0$.  Furthermore,
$\partial |\psi| = \bigcup_{[\phi] < [\psi]} |\phi|$: If $x \in
\partial |\psi| \subseteq X$, there is a unique $\sigma_i \in K_i$
such that $x \in \text{int}(|\sigma_i|)$, giving some $\phi$ with
$x \in |\phi|$ and $[\phi] \leq [\psi]$.  However $[\phi] <
[\psi]$ because $x$ is not in $\text{int}(|\psi|)$, so there is an
$i$ with $x$ not in $\text{int}(|\tau_i|)$, i.e., $\sigma_i \neq
\tau_i$.   Conversely, if $x \in \bigcup_{[\phi] < [\psi]}
|\phi|$, we have $x \in |\phi| \subseteq \partial |\psi|$ as
above.  This proves that $C_{\mathcal{K}}$ is a regular cell
structure on $X$.  Therefore, we have that $|\Delta
C_{\mathcal{K}}| \approx X$.
\end{proof}

\begin{cor}\label{reg2}
If $X$ is a $G$--polyhedron and $K$ is a nonequivariant triangulation
of $X$, then there is an admissible $G$--triangulation of $X$ which
is a subdivision of K.
\end{cor}
\begin{proof} For each $g \in G$, $gK$ is another
triangulation of $X$ because $G$ is acting linearly on the ambient
representation space.  Let $\mathcal{K} = \{gK \, | \, g \in G\}$.
Using the notation from the proof of \ref{commonsubdivision}, each
cell $[\phi]$ is given by a map $\phi\colon\thinspace G \rightarrow K$. Then
$|\phi| = \bigcap_{g \in G} g|\phi(g)| \subseteq X$. For $h \in
G$, define $(h\phi)(g) := \phi(h^{-1}g)$. This induces an
order-preserving $G$--action on $C_{\mathcal{K}}$ because, for any
$\phi$ and any $h$ in $G$, $|h\phi| = \bigcap_{g \in G}
g|\phi(h^{-1}g)| = \bigcap_{g\in G} hg|\phi(g)| = h\bigcap_{g \in
G} g|\phi(g)| = h|\phi|$.  Lastly, $C_{\mathcal{K}}$ is a
$G$--regular cell complex because the average of the vertices of
any cell is fixed by the cell's stabilizer.  The result follows by
Lemma \ref{Gcellcx}.
\end{proof}

\begin{lemma}\label{reg3}
Let (K, L) be an admissible $G$--triangulation of (X, Y) with L full
in K and let $(K_1, L_1)$ be a $G$--subdivision of (K, L).  Then
$\exists \; G$--derived subdivisions $K^{\prime}$ and
$K_1^{\prime}$ of K and $K_1$ near L and $L_1$ such that $|N_{K^{\prime}}(L)| = |N_{K_1^{\prime}}(L_1)|$.
\end{lemma}
\begin{proof} We follow the proof of the non-equivariant version, Lemma 3.7 in
\cite{RS82}, and clarify some details with the combinatorial
definition of derived subdivisions. We define a map $f_{L,K} = f \colon\thinspace X \rightarrow [0,1]$ by setting $f(v) = 0$ if $v \in L^0$ and $f(v) = 1$ if $v \in K^0 \setminus L^0$ gives a map on
the vertices of $K$ which can be linearly extended on simplices. Choose $\epsilon$ small enough so that
no vertex of $K_1 \setminus L_1$ is contained in
$f^{-1}[0,\epsilon]$.  Then choose $G$--derived subdivisions
$K^{\prime}$ and $K_1^{\prime}$ of $K$ and $K_1$ near $L$ and
$L_1$ respectively with all the new vertices $v_{\tau}$ lying in
$f^{-1}(\epsilon)$.  We can choose these vertices equivariantly
because $f$ is $G$--invariant and $K$ is admissible.  Now we will
show that $|N_{K^{\prime}}(L)| = f^{-1}[0,\epsilon] =
|N_{K_1^{\prime}}(L_1)|$.

The map $f$ takes values of $0$ or $\epsilon$ on all of the
vertices of $N_{K^{\prime}}(L)$ and $N_{K_1^{\prime}}(L_1)$, so both
of these neighborhoods are contained in $f^{-1}[0,\epsilon]$.  Now
let $x$ be a point in $ f^{-1}[0,\epsilon] \subset |K^{\prime}| =
|K_1^{\prime}|$; say $x$ is in the interior of the simplex $\sigma
\cup \{v_{\tau_1}, \ldots, v_{\tau_k}\}$ of $K^{\prime}$
(respectively $K_1^{\prime}$) as in the combinatorial definition
of a derived subdivision.  Then $x = s_0v_0 + \ldots + s_k v_k +
t_1 v_{\tau_1} + \ldots + t_m v_{\tau_m}$ where $\sigma = \{v_0,
\ldots, v_k\}$ and $s_0 + \ldots + s_k + t_1 + \ldots + t_m = 1$.
Suppose $\sigma$ is not in $L$ (resp. $L_1$).  Then $f(v_i) = 1$
(resp. $f(v_i) > \epsilon$) for $i = 0, \ldots, k$.  Meanwhile,
$f(v_{\tau_j}) = \epsilon$ for $j = 1, \ldots, m$.  We have then,
in both cases, that $f(x) > (s_0 + \ldots + s_k)\epsilon + (t_1 +
\ldots + t_m)\epsilon = \epsilon$.  This is a contradiction, and
$\sigma$ must be in $L$ (resp. $L_1$), yielding that $x$ is in
$|N_{K^{\prime}}(L)|$ (resp. $|N_{K_1^{\prime}}(L_1)|$).
\end{proof}

\begin{lemma}\label{links}
Let $v$ be a vertex of a $G$--simplicial complex $K$.  If $K^{\prime}$ is a derived $G$--subdivision of $K$ near $v$, then $|\lnk_K(v)|$ is $G_v$-homeomorphic to $|\lnk_{K^{\prime}}(v)|$.
\end{lemma}
\begin{proof}
Assume that the derived vertices are chosen in $f_{v,K}^{-1}(\epsilon)$ for some $\epsilon \in (0,1)$.  Then a point in $|\lnk_{K^{\prime}}(v)| = f^{-1}(\epsilon)$ is of the form $\epsilon u + (1-\epsilon)v$ where $u \in |\lnk_K(v)|$.  Mapping this point to $u$ gives the desired homeomorphism.
\end{proof}

\begin{theorem}\label{reg4}
If $N_1$ and $N_2$ are $G$--regular neighborhoods of Y in X, then
there exists a $G$--homeomorphism $h\colon\thinspace X \rightarrow X$ that maps
$N_1$ to $N_2$ and is the identity on Y.
\end{theorem}
\begin{proof} The proof mirrors that of the non-equivariant version, Theorem 3.8 in
\cite{RS82}. We have $N_i = |N_{K_i^{\prime}}(L_i)|$ for $i = 1,2$
as above, where $K_i$ are $G$--triangulations of a neighborhood of
$Y$ in $X$.  As in \cite{RS82}, $K_1$ and $K_2$ may be subdivided
and extended to triangulations of $|K_1| \cup |K_2|$. Applying
\ref{commonsubdivision}, there is a triangulation $K_0$ of $|K_1|
\cup |K_2|$ containing subdivisions of $K_1$ and $K_2$ and
inducing $L_0$, a common subdivision of $L_1$ and $L_2$. By
\ref{reg2}, $K_0$ may be assumed to be a $G$--triangulation.  By
\ref{reg3}, for $i = 1,2$, we can find $G$--derived subdivisions
$\tilde{K}_i$ and $K_0^i$ of $K_i$ and $K_0$ near $L_i$ and $L_0$
respectively such that $|N_{\tilde{K}_i}(L_i)| = |N_{K_0^i}(L_0)|$.
Therefore we have that $N_1 \approx_G |N_{\tilde{K}_1}(L_1)| =
|N_{K_0^1}(L_0)| \approx_G |N_{K_0^2}(L_0)| = |N_{\tilde{K}_2}(L_2)|
\approx_G N_2$ where each $G$--homeomorphism is induced by a
$G$--isomorphism of $G$--derived subdivisions.  Such an isomorphism
is given by changing the placement of each derived vertex
$v_{\tau}$ within the interior of a simplex $\tau$ touching $Y$
and fixing the placement of every other vertex in the subdivision,
so the corresponding $G$--homeomorphism is the identity on $Y$
itself.  Thus, the same is true of the composition of these
$G$--homeomorphisms.
\end{proof}

Let $Y \subset X$ be $G$--polyhedra.  A \textbf{$G$--collar} on $Y$
in $X$ is a PL $G$--embedding $c: Y \times I \rightarrow_G X$
such that $c(y, 0) = y$ and $c(Y \times [0,1))$ is an open
neighborhood of $Y$ in $X$.  Suppose that for every $a \in Y$ there are (closed, polyhedral) neighborhoods $U$ and $V$ of $a$ in $X$ and $Y$ respectively with $U \cap Y = V$, such that for any $g \in G$, $gU \cap U \neq \emptyset$ implies that $ga = a$ and $gU = U$.  Suppose further that $U \approx_{G_a} V \times I$ with $v \mapsto (v,0)$ on $V$.  Then we say that $Y$ is \textbf{locally $G$--collarable} in $X$ and we have that $GU \approx_G GV \times I$.  Local $G$--collarability is equivalent to $G$--collarability.

\begin{theorem}\label{reg5}
Suppose $Y \subset X$ is locally $G$--collarable.  Then
there is a $G$--collar on Y in X.
\end{theorem}
\begin{proof} The proof is essentially the same as that of the non-equivariant version, Theorem 2.25 in
\cite{RS82}, substituting local $G$--collars for local collars. 

Construct a new $G$--polyhedron $Z := X \cup Y \times [-1,0]$ by attaching a $G$--collar to $Y$ outside of $X$, identifying $Y \subset X$ with $Y \times \{0\}$.  We will construct a $G$--homeomorphism between $X$ and $Z$ which carries $Y$ to $Y \times \{-1\}$.  Then the preimage of $Y \times [-1,0]$ will be a $G$--collar on $Y$ in $X$.  

For each $a \in Y$, let $GV_a \times I$ be a local $G$--collar at $a$.  Using compactness, cover $Y$ with the interiors of finitely many $GV_{a_1}, \ldots, GV_{a_k}$.  Then for each $i$, we will define a $G$--homeomorphism $h_i \colon Z \rightarrow Z$ which maps the interior of $V_{a_i} \times \{0\}$ into the interior of $V_{a_i} \times [-1,0]$ and is the identity outside of $GV_{a_i} \times [-1,1]$.  To do this, we consider a $G$--triangulation $K_{a_i}$ of $GV_{a_i}$.  Taking the product of this triangulation with an interval defines a regular $G$--cell structure on $GV_{a_i} \times [-1,1]$.  We now equivariantly subdivide this cell complex by its order complex:  We must choose a new vertex $v_c$ in the interior of each cell $c$.  First, equivariantly select a point $y_{\sigma}$ in the interior of each simplex $\sigma \in K_{a_i}$.  For the cell $|\sigma| \times \{1\}$, choose $v_c = (y_{\sigma},1)$.  Likewise, for the cell $|\sigma| \times \{-1\}$, choose $v_c = (y_{\sigma},-1)$.  Finally, for the cell $|\sigma| \times [-1,1]$, choose $v_c = (y_{\sigma},0)$.  

Define a different subdivision simply by moving $v_c$ to $(y_{\sigma}, -\frac{1}{2})$ when $c$ intersects the relative interior of $GV_{a_i} \times [0,1]$ and making no change in the placement for $v_c$ otherwise.  Then the $G$--homeomorphism $h_i$ is given by mapping the first subdivision to the second, since they are both realizations of the same order complex.  Note that $h_i$ is the identity except on the relative interior of $GV_{a_i} \times [-1,1]$ where $h_i(y,t) = (y,s)$ with $s < t$.  Thus, $h_i$ maps all of the relative interior of $GV_{a_i} \times \{0\}$ into $GV_{a_i} \times (-1,0)$.

\begin{figure}
\includegraphics[scale=.7]{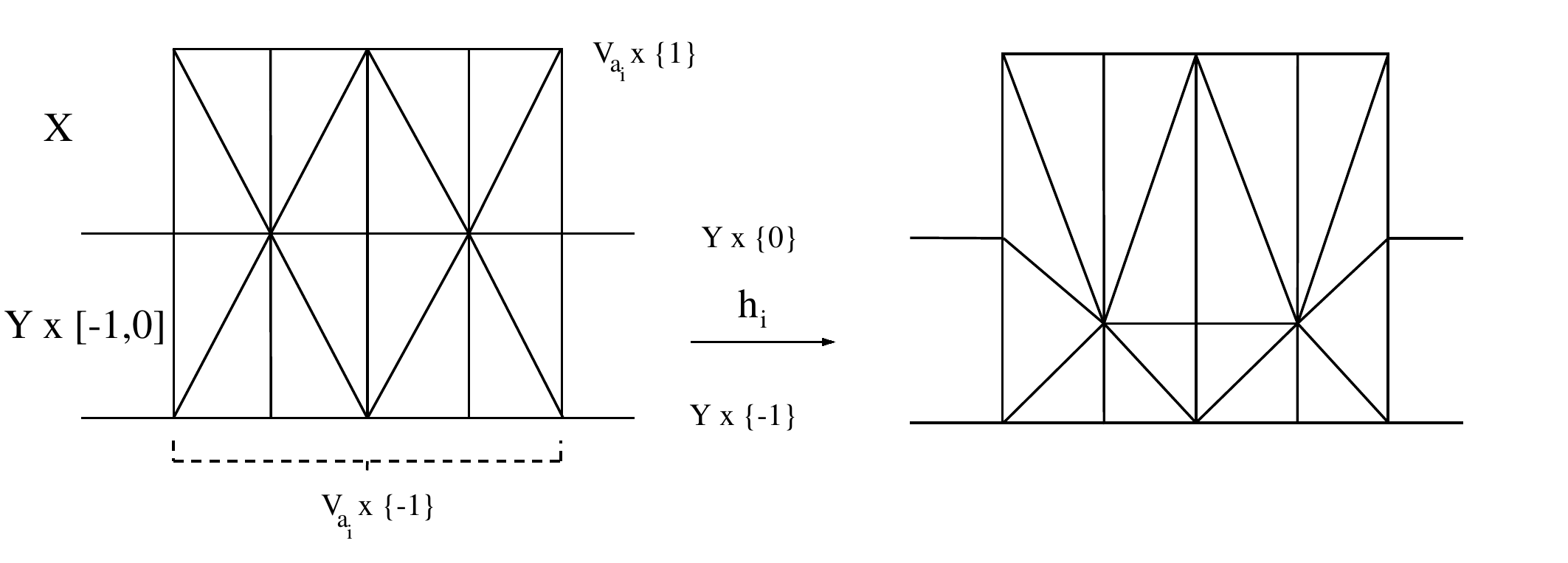}
\caption{Construction of $h_i$ \label{fig-Gcollar}}
\end{figure}

Now define $h$ to be the composition $h_k \circ \ldots \circ h_1$. Since the interiors of the local $G$--collars cover all of $Y$, $h$ maps all of $Y \times \{0\}$ into $Y \times (-1,0)$.  Let $K$ $G$--triangulate $Y$, and thus also $h(Y \times \{0\})$.  Then we consider $h(X) \cap [-1,0]$.  It has a regular cellular $G$--structure with a face poset isomorphic to that of $|K| \times [-1,0]$.  The cells in the former come in three types: simplices of $K$ triangulating $Y \times \{0\}$ (which correspond to the same in the latter complex), simplices of $K$ triangulating $h(Y \times \{0\})$ (which correspond to simplices of $K$ triangulating $Y \times \{-1\}$), and the intersection of $h(X)$ with cells $|\sigma| \times [-1,0]$ for $\sigma \in K$ (which correspond to the cells $|\sigma| \times [-1,0]$).  Therefore, $h(X) \cap [-1,0]$ is $G$--homeomorphic to $Y \times [-1,0]$, fixing $Y \times \{0\}$.  We extend this homeomorphism by the identity to the rest of $X$ to get the desired $G$--homeomorphism $\tilde{h} \colon X \rightarrow Z$ carrying $Y$ to $Y \times \{-1\}$.
\end{proof}

\begin{theorem}\label{reg7}
If $Y \subset X$ is locally $G$--collarable then a $G$--regular
neighborhood of Y in X is a $G$--collar.
\end{theorem}
\begin{proof} We have by \ref{reg5} that $Y$ is $G$--collarable.  Let $L = L
\times {0}$ be an admissible $G$--triangulation of $Y$. Let $K$ be
the $G$--triangulation of the $G$--collar $Y \times I$ as in
\ref{reg1}. Now choose a $G$--derived subdivision $K^{\prime}$ of
$K$ near $L$ such that all of the new vertices have
second coordinate $\frac{1}{2}$. Then $|N(L,K^{\prime})|
= Y \times [0,\frac{1}{2}] \approx_G Y \times I$.  The result now
follows from \ref{reg4}.
\end{proof}

We will also make use of the notion of bicollarability.  We say $Y \subset X$ is \textbf{$G$--bicollarable} in $X$ if there exists a $G$--embedding of $Y \times [-1,1]$ into $X$ with $(y,0) \mapsto y$ for all $y \in Y$ and $Y \times (-1,1)$ mapping to an open neighborhood of $Y$ in $X$.

\begin{theorem}\label{regnbhd-bicollar}
If $N = |N_{K^{\prime}}(L)|$  is a $G$--regular neighborhood of $Y$ in $X$, then $|\dot{N}_{K^{\prime}}(L)|$ is $G$--bicollarable in $X$.
\end{theorem}
\begin{proof}
By \ref{reg4}, it suffices to consider the case $N = f_{L,K}^{-1}[0,\epsilon]$ for some $\epsilon \in (0,1)$.  That is, the derived vertices $\{v_{\tau}\}$ of $K^{\prime}$ were chosen in $f^{-1}(\epsilon)$.  Let $0 < \epsilon_1 < \epsilon < \epsilon_2 < 1$.  Equivariantly, choose alternate derived vertices $\{v^1_{\tau}\}$ and $\{v^2_{\tau}\}$ in $f^{-1}(\epsilon_1)$ and $f^{-1}(\epsilon_2)$ respectively, giving derived $G$--subdivisions $K_1^{\prime}$ and $K_2^{\prime}$ of $K$ near $L$.  Then there are the natural homeomorphisms $h_i \colon |\dot{N}_{K^{\prime}}(L)| \rightarrow_G |\dot{N}_{K_i^{\prime}}(L)|$ given by sending each $v_{\tau}$ to $v^i_{\tau}$.  We now define a $G$--bicollar $C \colon |\dot{N}_{K^{\prime}}(L)| \times [-1, 1] \rightarrow_G cl( |N_{K_2^{\prime}}(L)| \setminus  |N_{K_1^{\prime}}(L)|)$ by setting
\begin{equation*}
C(x,t) = \left\{\begin{array}{lr}
|t|h_1(x) + (1 - |t|)x, & -1 \leq t \leq 0 \\
th_2(x) + (1 - t)x , & 0 \leq t \leq 1
\end{array}\right.
\end{equation*}\end{proof}

Note that this $G$--bicollarability can alternatively be expressed as $|\dot{N}_{K^{\prime}}(L)|$ being $G$--collarable in both $N$ and in $cl(X \setminus N)$.

For the remainder of the section, we wish to consider $G$--regular neighborhoods within manifolds.  For that purpose, we need to define a $G$--manifold.  We will consider $G$--polyhedra and $G$--complexes that are manifolds and that have particularly well-behaved $G$--actions.

First, consider an orthogonal representation $\rho \colon G \rightarrow O_n(\R)$.  We denote by $S(\rho)$ and $D(\rho)$ the unit sphere and disk respectively in the corresponding representation space.  Further, denote by $S_+(\rho)$ the hemisphere with last coordinate nonnegative and similarly for $D_+(\rho)$.  These have unique piecewise-linear structures coming from their smooth structures \cite{I00}.

We now inductively define a \textbf{combinatorial $G$--sphere}.  $S^0$ with a $G$--action is a combinatorial $0$--dimensional $G$--sphere.  An admissible simplicial $G$--complex $K$ with $|K|$ (PL) $G$--homeomorphic to $S(\rho)$ for some $\rho \colon G \rightarrow O_{n+1}(\R)$ is an $n$--dimensional combinatorial $G$--sphere if for every $v \in K^0$, $\lnk_K(v)$ is an $(n-1)$--dimensional combinatorial $G_v$--sphere, itself $G_v$--homeomorphic to $S(\R v^{\perp})$, where $\R v^{\perp}$ is the orthogonal complement in $\rho|_{G_v}$ of the trivial representation $\R v$.

Similarly, we may define a \textbf{combinatorial $G$--hemisphere} by substituting $S_+(\rho)$ and allowing links of vertices to be $n$--dimensional $G$--spheres or $G$--hemispheres in the above definition.  Finally, a \textbf{combinatorial $G$--disk} is simply the cone on a $G$--sphere or $G$--hemisphere with a $G$--fixed point.

An admissible simplicial $G$--complex $K$ is an $n$--dimensional \textbf{combinatorial $G$--manifold} if for every $v \in K^0$, $\lnk_K(v)$ is an $(n-1)$--dimensional combinatorial $G_v$--sphere or hemisphere. A $G$--polyhedron $M$ is an $n$--dimensional \textbf{(PL) $G$--manifold} (with boundary) if its admissible $G$--triangulations are $n$--dimensional combinatorial $G$--manifolds.  (If it is true for one, it is true for all.)  The boundary $\partial M$ is easily shown to be a $G$--submanifold.

\begin{prop}\label{submanifold}
Let $M$ be an $n$--dimensional $G$--manifold and $M_1$ be an $n$--dimensional $G$--invariant submanifold with $cl(\partial M_1 \cap \text{int } M)$ $G$--bicollarable in $M$.  Then $M_1$ is a $G$--manifold.
\end{prop}
\begin{proof}
Let $K$ be a $G$--triangulation of $M$ with subcomplexes $(K_1, L)$ triangulating $(M_1,\partial M_1)$.  We need to show that the link of any vertex $v \in K_1^0$ is a $G_v$--sphere or hemisphere.  We consider the case $v \in cl(\partial M_1 \cap \text{int } M)$.  The link of any other vertex of $K_1$ is the same in both $K$ and $K_1$.

Since $\partial M_1$ is $G$--bicollarable, we may consider a closed $G_v$--invariant neighborhood $|\st_L(v)| \times [-1,1] = U_v$ with $(x,0)$ identified with $x$ for all $x \in |\st_L(v)|$ and $|\st_{L_1}(v)| \times [0,1] \subset M_1$.  Triangulate $U_v$ in the following way:  First triangulate $|\lnk_L(v)| \times [-1,1]$.  Add to it by coning $|\lnk_L(v)| \times \{-1\}$ and $|\lnk_L(v)| \times \{1\}$ with $(v,-1)$ and $(v,1)$ respectively.  Let $J$ be this triangulation of $|\lnk_L(v)| \times [-1,1] \cup |\st_L(v)| \times \{-1,1\}$.  Then $J * \{v\}$ triangulates $|\st_L(v)| \times [-1,1]$.

Now subdivide $K$ so that it contains a subdivision of $J * \{v\}$ as a subcomplex.  Choose a derived $G_v$--subdivision $K^{\prime}$ near $v$.  From the construction in the proof of \ref{reg3}, we may assume $|\st_K^{\prime}(v)|$ is an $\epsilon$--neighborhood of $v$ in $|J * \{v\}|$.  Thus, by \ref{links}, $|\lnk_L(v)| \times [-1,1] \cup |\st_L(v)| \times \{-1,1\}$ is $G_v$--homeomorphic to $|\lnk_K^{\prime}(v)|$, which we know is an $(n-1)$--$G_v$--sphere or hemisphere since $M$ is a $G$--manifold.

Consider the point $w = (v,1)$.  Its stabilizer in $G_v$ is all of $G_v$.  Hence, $\lnk_J(w)$ is an $(n-2)$--$G_v$--sphere or hemisphere.  But $|\lnk_J(w)|$ is $G_v$--homeomorphic to $|\lnk_L(v)|$, giving us that $|J|$ is the suspension of $|\lnk_L(v)|$ by $w$ and $(v,-1)$, and thus a $G_v$--sphere or hemisphere.  In conclusion, $|\lnk_L(v)| \times [0,1] \cup |\st_L(v)| \times \{1\}$, which is $G_v$--homeomorphic to $|\lnk_{K_1}(v)|$, is a $G_v$--hemisphere.  (Note that when $\lnk_L(v)|$ is a $G_v$--hemisphere, $|\lnk_L(v)| \times [0,1] \cup |\st_L(v)| \times \{1\}$ is actually a quadrant of a $G_v$--sphere where the two nonnegative coordinates give trivial subrepresentations.  This is easily seen to be $G_v$--homeomorphic to a $G_v$--hemisphere.)
\end{proof}

Now assume for the remainder of the section that $Y \subset M$ is
a polyhedron and $M$ is an $n$--$G$--manifold.  The following is a direct result of \ref{submanifold}, \ref{regnbhd-bicollar}, and the non-equivariant Proposition 3.10 from \cite{RS82}

\begin{prop}\label{reg8}
A $G$--regular neighborhood N of Y in M is a $G$--manifold with boundary. If
$Y \subset$ int M, $\partial N = |\dot{N}_{K^{\prime}}(L)|$ where L
and $K^{\prime}$ are as in the definition of regular neighborhood.
\end{prop}

We omit the proofs of the following two results since they are exactly the same as the non-equivariant versions (Theorem 3.11 and Corollary 3.18) in \cite{RS82}, only utilizing $G$--bicollarability in place of collarability of boundaries of manifolds.

\begin{theorem}[Simplicial $G$--Neighborhood Theorem]\label{reg9}
Let Y be a $G$--polyhedron in int $M^n$ and N be a
$G$--neighborhood of Y in int $M^n$.  Then N is $G$--regular if and only
if
\begin{itemize}
\item[(i)]N is an $n$--manifold with boundary $\partial N$ $G$--bicollarable in $M$, \item[(ii)]there
    are admissible $G$--triangulations (K, L, J) of (N, Y,
    $\partial N$) with L full in K, $K = N_K(L)$ and $J =
    \dot{N}_K(L)$
\end{itemize}
\end{theorem}

\begin{cor}\label{reg10}
Suppose $N_1 \subset \text{int } N_2$ are two $G$--regular
neighborhoods of Y in int M.  Then cl($N_2 \setminus N_1$) is a
$G$--collar of $\dot{N}_1$.
\end{cor}

Let $Y \subset X$ be polyhedra such that $X = Y \cup D^m$ and $Y
\cap D^m = D^{m-1}$ where $D^m$ and $D^{m-1}$ are disks of
dimensions $m \geq 1$ and $m - 1$ respectively and there is a PL
homeomorphism $D^m \rightarrow D^{m-1} \times I$ with $D^{m-1}$ mapping by the identity to
$D^{m-1} \times \{0\}$. Then we say there is an ($m$--dimensional) \textbf{elementary
collapse} of $X$ onto $Y$. If there is a sequence $X = X_0, X_1,
X_2, \ldots, X_k = Y$ of polyhedra such that there is an
elementary collapse of $X_{i-1}$ onto $X_i$, we say $X$
\textbf{collapses} to $Y$, or $X \searrow Y$. In particular, if $X
= |K|$ and $Y = |L|$ with $K$ collapsing simplicially to $L$, then
$X$ collapses to $Y$.  Note that the dimensions of the individual elementary collapses in a sequence are allowed to vary.  We say that a collapse is $m$--dimensional if every elementary step has dimension $\leq m$.

Consider now the case that $G$ is a finite group and $X$ and $Y$
are $G$--polyhedra with $X = Y \cup GD^m$ where $D^m$ (as well as
$gD^m$ for each $g \in G$) is a disk as in the definition of an
elementary collapse. Suppose we also have that: (1) $gD^m \neq D^m$
implies that $gD^m \cap D^m \subseteq Y$ and (2) there exists a point $y \in D^{m-1}$ fixed by the stabilizer $G_{D^m} = H$ such that $D^{m-1}$ is $H$--homeomorphic to a cone with apex $y$ on some $H$--complex, and (3) $D^m$ is $H$--homeomorphic to $D^{m-1} \times I$, then we say there is an \textbf{elementary $G$--collapse} from $X$ to $Y$.  (Note then that any point $x \in \{y\} \times (0,1]$ will have stabilizer $G_x = H$.)  A sequence of these is called a $G$--collapse, denoted $X \searrow_G Y$.

A simplicial collapse from admissible $K$ to $L$ gives rise to a
$G$--collapse if and only if whenever $(\sigma \lessdot \tau)$ is
in its corresponding Morse matching, so too is $(g\sigma \lessdot
g\tau) \; \forall \, g \in G$.

An elementary collapse from an n-manifold $M$ to another
n-manifold $M_1$ where $M = M_1 \cup D^n$ and $D^{n-1} \subset
\partial M_1$ and $D^{n} \setminus D^{n-1} \subset \partial M$ is
called an \textbf{elementary shelling}.  A sequence of such
collapses is a \textbf{shelling}.  Note that every elementary step must be $n$--dimensional.

The equivariant version of shelling requires some additional conditions.  Let $M_1 \subset M$ be $n$--manifolds with an elementary $G$--collapse from $M = M_1 \cup GD^n$ to $M_1$ such that: (1) $D^n \cap M_1 = D^{n-1}$ lies in a $G$--collarable subpolyhedron $W \subset \partial M_1$, (2) under the $G_{D^n}$--triangulation $K = \{y\} * L$ of $D^{n-1}$, if $gD^n \neq D^n$, then $gD^n \cap D^n \subset |L| \times \{0\}$, and (3) $y \in \partial D^{n-1}$ implies that, in the $G$--collar on $W$, $(\{y\} * \partial |L|) \times I \subset \partial M$.  If these three extra conditions are satisfied, this elementary $G$--collapse is called an \textbf{elementary $G$--shelling}.  A sequence of these is called a \textbf{$G$--shelling}.

\begin{lemma}\label{reg11}
If M $G$--shells to $M_1$, then there is a $G$--homeomorphism $h\colon\thinspace M
\rightarrow M_1$ which is the identity outside an arbitrary neighborhood of $M \setminus M_1$.
\end{lemma}
\begin{proof} As in the corresponding proof of Lemma 3.25 in \cite{RS82}, we need only to
consider the case of an elementary $G$--shelling.  Let $M = M_1 \cup GD^n$ give the elementary $G$--shelling.  Denote $G_{D^n}$ by $H$.  Let $K = \{y\} * L$ be the $H$--triangulation of $D^{n-1}$ from the definition of elementary $G$--collapse.

Choose a $G$--collar on $GD^{n-1}$ in $M_1$ within the given neighborhood of $M \setminus M_1$.   We may consider the disk $D^{n-1} \times [-1,1]$ with $D^n = D^{n-1} \times [-1,0]$, $E^n = D^{n-1} \times [0,1] \subset M_1$, and $D^{n-1} = D^{n-1} \times \{0\}$.  Then if $D^n \neq gD^n$, we have that $D^{n-1} \times [-1,1]$ may only intersect $gD^{n-1} \times [-1,1]$ in $|L| \times [0,1]$.  We will define an $H$--homeomorphism from $D^{n-1} \times [-1,1]$ to $E^n$ which is the identity on $|L| \times [0,1] \cup |K| \times \{1\}$; such a homeomorphism can then be extended, first equivariantly to all of $G(D^{n-1} \times [-1, 1])$ and then by the identity to the rest of $M$.  This last extension is possible because either $|L| \times I = \partial D^{n-1} \times I$ or $cl(\partial D^{n-1}) \setminus |L|) \times I = |\{y\} * \partial L| \times I \subset \partial M$.

\begin{figure}
\includegraphics[scale=.7]{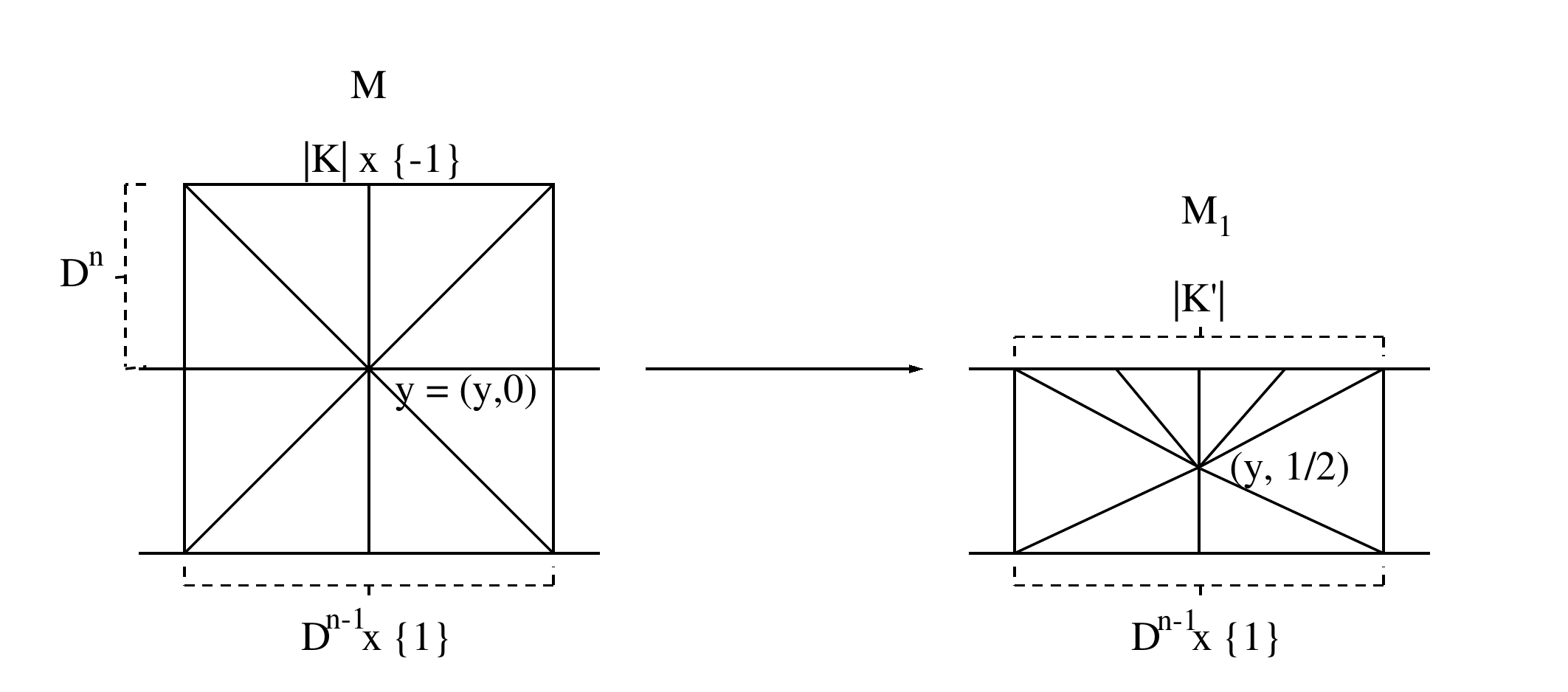}
\caption{Shelling homeomorphism \label{fig-Gshelling}}
\end{figure}

Let $K^{\prime}$ be a derived $H$--subdivision of $K$ near $y$.  We have that $|K| \times \{-1\} \cup |L| \times (-1,0)$ is $H$--homeomorphic to $|K^{\prime}| = |N_{K^{\prime}}(y)| \cup |N_{K^{\prime}}(L)|$ because they are both $H$--homeomorphic to $|K|$ with an $H$--collar attached outside to $|L|$.  Therefore, we have an $H$--homeomorphism from $D^{n-1} \times \{-1,1\} \cup |L| \times [-1,1]$ to $D^{n-1} \times \{0,1\} \cup |L| \times [0,1]$.  Coning the two polyhedra with $(y,0)$ and $(y,\frac{1}{2})$ gives the desired $H$--homeomorphism from $E^n \cup D^n$ to $E^n$.
\end{proof}

\begin{theorem}\label{reg12}
Suppose $Y \subseteq X$ are $G$--polyhedra in a $G$--manifold $M$.  If $X \searrow_G Y$, then a $G$--regular neighborhood of $X$ $G$--shells to a $G$--regular neighborhood of $Y$ in $M$.
\end{theorem}
\begin{proof}
We follow the proof of the non-equivariant version, Theorem 3.26 in \cite{RS82}, checking that the conditions of $G$--shelling are satisfied.   The proof uses induction on the dimension of the collapse from $X$ to $Y$.

Suppose that the theorem holds when the $G$--collapse is $(m-1)$--dimensional.  We now consider the case where there is an $m$--dimensional elementary $G$--collapse from $X$ to $Y$.  Let $X = Y \cup GD^m$, with $Y \cap D^m = D^{m-1} \times \{0\}$ where $D^m \approx_{G_{D^m}} D^{m-1} \times I$.   For simplicity, we will from now on denote the subgroup $G_{D^m}$ by $H$.

Let $K$ be an admissible $G$--triangulation of $M$ with full subcomplexes $L_2 \leq L_1$ triangulating $Y$ and $X$ respectively.  Denote by $J$ the subcomplex triangulating $Z = D^{m-1} \times \{1\} \subset D^m$, and by $GJ$, the resulting $G$--triangulation of its $G$--orbit, $GZ$.  Finally, let $y$ be the apex in the $G_{D^m}$--cone structure of $D^{m-1}$.  Note then that $\{y\} \times I$ is fixed pointwise by $H$, and any point $(y,t)$ with $t >0$ has stabilizer exactly $H$.  Let $x = (y, \frac{1}{2})$.

As in \cite{RS82}, by breaking up the collapsing into smaller steps, we may assume that there are no vertices of $K^0$ in $D^{m-1} \times (0,1)$.

Now we choose a derived $G$--subdivision $K^{\prime}$ of $K$ near $L_2 \cup GJ$.  Choose the derived vertices for simplices in $L_1 \setminus (L_2 \cup GJ)$ in $Gp^{-1}(\frac{1}{2})$ ensuring that $x$ is one of them, and denote by $L^{\prime}$ the new triangulation of $X$.  Then $N_{K^{\prime}}(L^{\prime})$ gives a $G$--regular neighborhood of $X$, which is the union of $N_{K^{\prime}}(L_2)$ and $N_{K^{\prime}}(GJ)$, $G$--regular neighborhoods of $Y$ and $GZ$ respectively.  There is an $(m-1)$--dimensional $H$--collapse from $|J|$ to $(y, 1)$, so the induction hypothesis and \ref{reg11} together imply that $|N_{K^{\prime}}(J)|$ is an $n$--dimensional $H$--disk.  Let $E^n = |N_{K^{\prime}}(J)|$.  As an $H$--disk with $x$ an $H$--fixed point, $E^n$ is seen to be $H$--homeomorphic to $|\st_{\dot{N}_{K^{\prime}}(J)}(x)| \times I$.

\begin{figure}
\includegraphics[scale=.7]{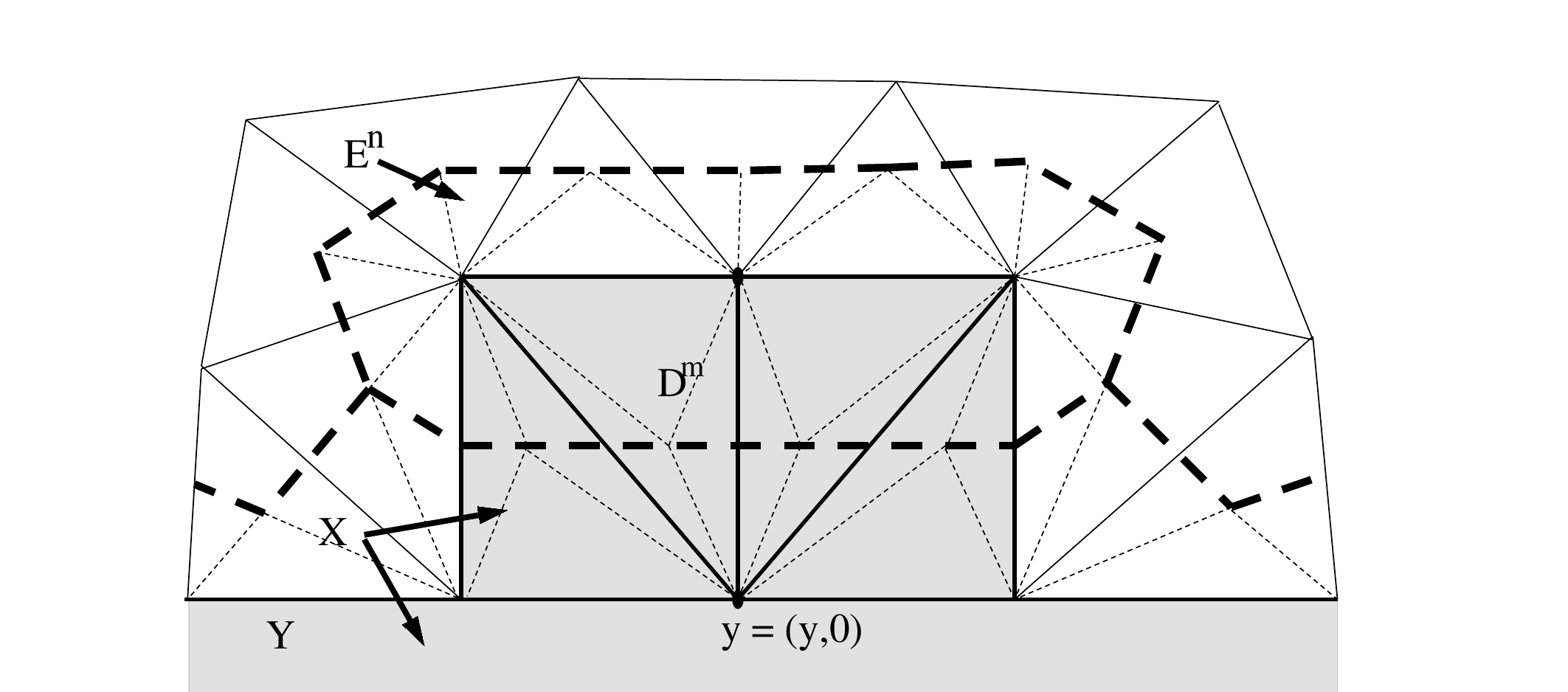}
\caption{Regular neighborhood shelling  \label{fig-xtoy}}
\end{figure}

We will show that if $N_{K^{\prime}}(gJ) \neq N_{K^{\prime}}(J)$, the two subcomplexes must be disjoint.  For such a $g$, suppose there exists a vertex $v = v_{\tau} \in N_{K^{\prime}}(gJ) \cap N_{K^{\prime}}(J)$.  (Note that it must be a derived vertex since $gJ$ and $J$ are themselves disjoint.)  Then $\tau \in K$ contains vertices $u$ and $w$ of $gJ$ and $J$ respectively.  Thus, $\rho = \{u, w\} \in L_1$ since $L_1$ is a full subcomplex of $K$, but $|\rho|$ is not contained in $Y$ and it is not contained in $GD^m$ since a simplex in $D^m$ is may only contain vertices from $L_2$ and $J$, not $gJ$.  This contradicts $X = Y \cup GD^m$, so $N_{K^{\prime}}(gJ) \cap N_{K^{\prime}}(J)$ must be empty.   Since we have shown that $gE^n \neq E^n$ implies $gE^n \cap E^n = \emptyset$, it must be true that $G_{E^n} = H$.

We next prove that $|N_{K^{\prime}}(J)| \cap |N_{K^{\prime}}(L_2)|$ is an $(n-1)$--disk $E^{n-1}$ which is $H$--homeomorphic to $|\st_{\dot{N}_{K^{\prime}}(J)}(x)|$, giving that $E^n$ is $H$--homeomorphic to $E^{n-1} \times I$ as required.  To see this, we show that $E^{n-1}$ is an $H$--regular neighborhood of $D^{m-1} \times \{\frac{1}{2}\}$ in the $(n-1)$--$H$--manifold $|\dot{N}_{K^{\prime}}(J)|$, so that we may again invoke the induction hypothesis for $(m-1)$--dimensional collapses and \ref{shelling-homeo} (since $D^{m-1} \times \{\frac{1}{2}\}$ $H$--collapses to $x$ and an $H$--regular neighborhood of $x$ is the desired star of $x$). 

Let $P$ be the subcomplex of $K^{\prime}$ triangulating $D^{m-1} \times \{\frac{1}{2}\}$ and let $Q = \dot{N}_{K^{\prime}}(J)$ for brevity.  The claim then is that $N_Q(P) = N_{K^{\prime}}(L_2) \cap N_{K^{\prime}}(J)$.  

Let $\sigma \in  N_{K^{\prime}}(L_2) \cap N_{K^{\prime}}(J)$, we easily see that $\sigma$ cannot intersect $L_2^0$ or $J^0$ and must consist only of derived vertices of the form $v_{\rho}$.  Then there must exist $u \in L_2^0$ and $w \in J^0$ such that $\sigma \cup \{u\}$ and $\sigma \cup \{w\}$ are both simplices of $K^{\prime}$.  This implies that there exists $v_{\rho} \in \sigma$ for some $\rho \in K$ containing both $u$ and $w$.  But then $\{u, w\} \in L_1$ due to the fullness of $L_1$.  Thus, $v_{\{u,w\}}$ is in $P$ and can be added to $\sigma$, so $\sigma \in N_Q(P)$.  Hence, we have $N_{K^{\prime}}(L_2) \cap N_{K^{\prime}}(J) \subseteq N_Q(P)$.

For the other inclusion, if $\sigma$ is in $N_Q(P)$, it means that there is a $v_{\tau} \in P^0$ such that $\sigma \cup \{v_\tau\}$ is in $Q$ for some $\tau$ which contains vertices from both $L_2$ and $J$.  We note again that $\sigma$ consists only of derived vertices since it is in $Q = \dot{N}_{K^{\prime}}(J)$, so let $\rho$ be the minimal face such that $v_{\rho} \in \sigma \cup \{\tau\}$.  Then $\rho \leq \tau$, so we have that $\rho \in L_1$.  Since $\rho$ was subdivided, it must contain some vertex $u \in L_2^0$.  Therefore, $u$ may be added to $\sigma$ to get a simplex of $K^{\prime}$ intersecting $L_2^0$, i.e., $\sigma \in N_{K^{\prime}}(L_2)$, and it is already in $Q \subset N_{K^{\prime}}(J)$.  This proves that $N_{K^{\prime}}(L_2) \cap N_{K^{\prime}}(J) = N_Q(P)$.  This finishes the proof that $E^{n-1}$ is an $H$--regular neighborhood of $D^{m-1} \times \{\frac{1}{2}\}$ in $|Q|$ and therefore $H$--homeomorphic to the $(n-1)$--disk $|\st_Q(x)|$ as explained.

Observe that $GE^n \cap |N_{K^{\prime}}(L_2)| \subset |\dot{N}_{K^{\prime}}(L_2)|$, which is $G$--collarable in $|N_{K^{\prime}}(L_2)|$ by \ref{regnbhd-bicollar}.

There is one remaining condition to check for this to be a $G$--shelling.  Write $E^{n-1} = |\st_Q(x)|$.  Then we must verify that $x \in \partial E^{n-1}$ implies that within the $G$--collar on $|\dot{N}_{K^{\prime}}(L_2)|$ in $|N_{K^{\prime}}(L_2)|$, $| \{x\}* \partial \lnk_Q(x)| \times I \subset \partial |N_{K^{\prime}}(L^{\prime})|$.  It suffices for us to show that every simplex of $\{x\} * \partial \lnk_Q(x)$ lies on $\partial M$ because the $G$--collar is given by moving derived vertices around with simplices of $K$.  Thus, if a simplex $\sigma$ consisting only of derived vertices lies on $\partial M$, then $\sigma \times I$ lies on $\partial M$, and hence also on $\partial |N_{K^{\prime}}(L^{\prime})|$.

Since $E^{n-1}$ is an $H$--regular neighborhood of $x \in |Q|$, if $x \in \partial E^{n-1}$, then $x \in \partial |Q| = |Q| \cap \partial M$.  Thus, $x \in \partial M$, forcing it to also belong to $\partial |N_{K^{\prime}}(L^{\prime})|$.  Likewise, any simplex of $\sigma \in \{x\} * \partial \lnk_Q(x)$ containing $x$ lies in $\partial E^{n-1}$ but not in $\dot{N}_Q(x)$, forcing $\sigma$ to be in $\partial M$ and therefore $\partial  |N_{K^{\prime}}(L^{\prime})|$.  Hence, we have proven the final condition that this constitutes a $G$--shelling of a $G$--regular neighborhood of $X$ to a $G$--regular neighborhood of $Y$.
\end{proof}

\begin{cor}\label{reg13}
If Y $G$--collapses to a point, then any $G$--regular neighborhood of Y (in a $G$--manifold) is a
$G$--disk.
\end{cor}

\begin{cor}\label{reg14}
A collapsible $G$--manifold is a $G$--disk.
\end{cor}

\begin{cor}\label{reg15}
If $X \subset \text{int } M$ and $X \searrow_G Y$ then a $G$--regular
neighborhood of X in M is a $G$--regular neighborhood of Y in M.
\end{cor}
\begin{proof} Let $N_1$ and $N_2$ be $G$--regular neighborhoods of $X$ and
$Y$ respectively in $M$.  By \ref{reg12}, $N_1$ $G$--shells to
$N_2$, and so by \ref{reg11}, they are $G$--homeomorphic fixing
$Y$. The homeomorphism carries any $G$--triangulation of $N_2$ to a
$G$--triangulation of $N_1$.  Therefore, by \ref{reg9}, $N_1$ is a
$G$--regular neighborhood of $Y$ in $M$.
\end{proof}

Our final goal for this section is the following result that enables us to recognize $G$--regular neighborhoods from $G$--collapses.

\begin{theorem}[Collapsing Criterion for $G$--Regular Neighborhoods]\label{reg16}
Let N be a $G$--neighborhood of Y in int M. Then N is $G$--regular if
and only if
\begin{itemize}
\item[(i)]N is an $n$--manifold with $\partial N$ $G$--bicollarable in $M$, \item[(ii)]N
    $\searrow_G$ Y.
\end{itemize}
\end{theorem}
\begin{proof} The proof follows exactly the non-equivariant
case (Corollary 3.30 in \cite{RS82}), but we include it here nonetheless because this theorem is a key tool in proving our main result.  For the first implication, suppose $N = |K|$ is regular with $K$ an admissible simplicial neighborhood of $L$; then (i) follows from \ref{reg8} and \ref{regnbhd-bicollar}. Recall the map $f = f_{L,K}$ from the proof of
\ref{reg3}. Choose $\epsilon \in (0,1)$ and choose a $G$--derived
$K^{\prime}$ of $K$ near $L$ with all of the new vertices lying in
$f^{-1}(\epsilon)$ for a given $\epsilon \in (0,1)$ to obtain $N_1
= |N_{K^{\prime}}(L)| = f^{-1}[0,\epsilon]$, a $G$--regular
neighborhood of $Y$ in $M$. We have that $N \approx_G N_1$, and we
have a $G$--cell structure on $N_1$ whose cells are obtained by
intersecting the interior simplices of $K$ with $f^{-1}(0)$,
$f^{-1}(\epsilon)$, or $f^{-1}[0,\epsilon]$. We may collapse, along with its
orbit, each cell $\sigma \cap f^{-1}[0,\epsilon]$, $\sigma \in K
\setminus \dot{N}_K(L)$ from its face $\sigma \cap f^{-1}(\epsilon)$ in
order of decreasing dimension.  That this is a $G$--collapse follows from the admissibility of $K$.

For the other implication, suppose we have $N$ satisfying conditions (i) and (ii).  Let $C = \partial N \times [-1,1]$ be a $G$--bicollar with $\partial N = \partial N \times \{0\}$.  Then let $N_1 = N \cup (\partial N \times [0,\frac{1}{2}])$, which constitutes a $G$--regular neighborhood of $N$ in $M$ because we can triangulate it to be a simplicial neighborhood.  Therefore, by \ref{reg15}, since $N \searrow_G Y$, $N_1$ is also a $G$--regular neighborhood of $Y$.  But we can define a $G$--homeomorphism on $C$ fixing $\partial N \times \{-1,1\}$ and carrying $\partial N \times \{\frac{1}{2}\}$ to $\partial N \times \{0\}$.  We can extend this by the identity to all of $M$, mapping $N_1$ to $N$, showing that the latter is also a $G$--regular neighborhood of $Y$ in $M$.
\end{proof}

\section{Main Results}
In this section, $G$ is the group $\{\pm 1\}$ unless otherwise
noted. Working inside $\Hom(P_4 \setminus \{3\},K_n)$ ($G$--homeomorphic to the $G$--manifold $S^{n-2} \times S^{n-2}$), we show that $\Hom_{\{3\}}(P_4, K_n)$ is a
$G$--regular neighborhood of the diagonal $\Hom(K_2,K_n)$ using the collapsing criterion, i.e.,
we show that it is a manifold of the correct dimension and that it (simplicially) $G$--collapses to the
diagonal.  Recall from Section 1 how we represent elements of the posets in question as arrays
whose entries $A, B, C, D$ are nonempty subsets of $\{1, \ldots, n\}$.

Define
\begin{equation*}
  M := \{\phi = \left(
   \begin{array}{ccc}
    A & B \\
    C & D \\
   \end{array}\right) \quad
   | \quad A\cap B = \emptyset, \quad C\cap D = \emptyset \}
\end{equation*}
\begin{equation*}
  K := \{\phi \in M \quad | \quad B\cup D \neq  \{1, \ldots, n\} \}
\end{equation*}
\begin{equation*}
  L := \{\phi \in K \quad | \quad A \cap C = \emptyset \}
\end{equation*}
\begin{equation*}
  S := \{\phi \in K \quad | \quad A = C, \quad B = D\}
\end{equation*}

We reiterate that $M$, $K$, and $L$ are the face posets of the
$G$--regular cell complexes $\Hom(P_4 \setminus \{3\}, K_n)$,
$\Hom_{\{3\}}(P_4, K_n)$, and $\Hom_{\{3\}}(C_5, K_n)$
respectively, and $S$ that of the diagonal $\Hom(K_2,K_n)$.
By passing to order complexes, we obtain that $\Delta S$ and $\Delta L$ are full $G$--subcomplexes of $\Delta K$, which is a full $G$--subcomplex of $\Delta M$, and they are all admissible.  Our goal is to show that $|\Delta K|$ is a $G$--regular neighborhood of $|\Delta S|$ whose boundary is $|\Delta L|$.

\begin{prop}\label{main1}
$\Delta K$ is a $(2n - 4)$--manifold with boundary $\Delta
L$.
\end{prop}
\begin{proof}
We show that the link of an element of $K$ is a sphere or a
disk of dimension $(2n-5)$. For any $\phi = \left(
  \begin{array}{ccc}
   A & B \\
   C & D \\
  \end{array}\right) \in K$, $\text{lnk}_{\Delta K}(\phi) = \Delta K_{<\phi} \ast
\Delta K_{>\phi}$. For any $\phi \in K$, we obtain an element of
its lower link by deleting proper subsets from each of $A$, $B$,
$C$, and $D$, at least one of which is nonempty. Therefore,
$K_{<\phi}$ is isomorphic to the face poset of
$\partial\Delta A \ast\partial\Delta B\ast\partial\Delta C \ast\partial\Delta D$,
yielding that $\Delta K_{<\phi}$ is a combinatorial sphere of
dimension $|A| + |B| + |C| + |D| - 5$.  (Recall that, if $A$ is an unordered set, $\Delta A$ is the full simplex having $A$ as its vertex set, whereas, if $P$ is a poset,  $\Delta P$ is its order complex.)

When $\phi \in K\setminus L$, we show that $\Delta K_{>\phi}$ is a
sphere of dimension $2n-|A|-|B|-|C|-|D|-1$, yielding that
$\text{lnk}_{\Delta K}(\phi)$ is a sphere of dimension $2n-5$. For
any $\phi^{\prime} = \left(
  \begin{array}{ccc}
   A^{\prime} & B^{\prime} \\
   C^{\prime} & D^{\prime} \\
  \end{array}\right) \in M \text{ such that } \phi^{\prime} > \phi$, we have that
\begin{equation*}
\emptyset \neq A\cap C \subseteq A^{\prime} \cap C^{\prime}
\subseteq (B^{\prime} \cup D^{\prime})^c
\end{equation*}
so that $\phi^{\prime} \in K$. Thus, to obtain an element of the upper link of
$\phi$, any element of $(A \cup B)^c$ can be added to either $A$ or
$B$, but not to both, and similarly for elements of $(C \cup
D)^c$.  As a consequence, we have that $K_{>\phi}$ is isomorphic to the face poset of
$\ast_{i=1}^m S^0$ where $m = |(A\cup B)^c|+|(C\cup D)^c|$, and
therefore  $\Delta K_{>\phi}$ is a sphere of dimension
$2n-|A|-|B|-|C|-|D|-1$ as claimed.

In the case where $\phi \in L$, we claim that $\Delta K_{>\phi}$
is a disk of dimension $2n-|A|-|B|-|C|-|D|-1$, meaning that
$\text{lnk}_{\Delta K}(\phi)$ is a disk of dimension $2n-5$.  This
will finish the proof that $\Delta K$ is a manifold and $\Delta L$
is its boundary. To see that $K_{>\phi}$ is the face poset of a
subcomplex of a join of spheres, we consider the various types of
elements that we can add to one or more of $A$, $B$, $C$, and $D$
to obtain a larger element of $K$.
\begin{enumerate}
\item An element of $A\cap D$ cannot be added anywhere (while
remaining in $M$).  The same is true for elements of $B\cap D$ and
$B\cap C$. Thus, these elements contribute nothing to the upper
link. \item An element of $B \setminus (C\cup D)$ can be added to
$C$ or to $D$; doing so will give us something in $K$, since the
element was already in $B \cup D$. Thus, each of these elements
contributes a copy of $S^0$ to the join of spheres. Similarly,
each element of $D \setminus (A\cup B)$ contributes a copy of
$S^0$ to the join. \item An element of $A\setminus D$ can be added
to $C$ or to $D$, contributing a copy of $S^0 = \{\pm 1\}$ to the
join with $+1$ indicating that the element was added to $C$ and
$-1$ indicating $D$. Similarly, an element of $C\setminus B$ can
be added to $A$ $(+1)$ or to $B$ $(-1)$. Adding elements of Type
$3$ to $B$ or $D$ could produce something not in $K$. \item An
element of $(A\cup B \cup C \cup D)^c$ can be added to $A$ or $B$
(but not both) and, at the same time, to $C$ or $D$ (but not
both).  This contributes a copy of $S^1 = \{\pm 1\}\ast\{\pm 1\}$
(treated as a single coordinate) to the join of spheres with the
$+1$'s corresponding to $A$ and $C$ and the $-1$'s corresponding
to $B$ and $D$. As with Type $3$, adding this type of element to
$B$ or $D$ could yield something not in $K$.
\end{enumerate}
To ensure that we remain in $K$, there must be an element of
$(B\cup D)^c$ which is not added to $B\cup D$.  In terms of
coordinates, this means there must be at least one coordinate
corresponding to Type $3$ or $4$ above that has no $-1$'s.

Before proceeding, we define \begin{equation*}F_{k,l}\subseteq
(\ast_{i=1}^kS^1)\ast(\ast_{j=1}^l\{\pm 1\})\end{equation*} to be
the subcomplex whose simplices have at least one coordinate from
the join with no $-1$'s.  (Note that, as before, each copy of $S^1
= \{\pm 1\}\ast\{\pm 1\}$ is regarded as a single coordinate.)  We
will prove a lemma (\ref{main2}) stating that $F_{k,l}$ is a disk of dimension
$2k + l - 1$. Assuming that result, since $K_{>\phi}$ is isomorphic to the face poset of
$(\ast_{i=1}^mS^0)\ast F_{k,l}$ where $k = |(A \cup B \cup C
\cup D)^c|$, $l = |A \setminus D| + |C \setminus B|$, and $m = |B
\setminus (C\cup D)| + |D \setminus (A\cup B)|$, we have that
$\Delta K_{>\phi}$ is a disk of dimension
\begin{align*}m-1+2k+l &= |B\setminus (C\cup D)| + |D\setminus (A\cup B)| \\
& \qquad + 2|(A\cup B\cup C\cup D)^c| + |A\setminus D| + |C\setminus B|-1 \\
&= 2n - 2|A| - 2|B| - 2|C| - 2|D| \\
& \qquad + 2|A\cap D| + 2|B\cap C| + 2|B\cap D| \\
& \qquad + |B\setminus (C\cup D)| + |D\setminus (A\cup B)| \\
& \qquad + |A\setminus D| + |C\setminus B| - 1\\
&= 2n - |A| - |B| - |C| - |D| - 1
\end{align*}
as we had claimed.
\end{proof}

\begin{lemma}\label{main2}
For $k, l \in \N \text{ such that } 2k + l - 1 \geq 0$, $F_{k,l}$
is a disk of dimension $2k + l - 1$.
\end{lemma}
\begin{proof} We proceed by induction on the dimension, $2k + l -
1$. In the initial case, $F_{0,1}$ has a single $S^0$ coordinate which must be $+1$, so it is a single point, i.e. a disk of dimension $0$. To prove
$F_{k,l}$ is a disk, we will show that it is a $(2k + l -
1)$--manifold, show it collapses to a vertex, and then apply
Corollary \ref{reg14}.  There are four
types of vertices whose links we need to consider:
  \begin{enumerate}
  \item $+1$ coming from one of the $k$ $S^1$ coordinates has as
  its link $F_{k-1,l+1}$, a $(2k+l-2)$--disk by induction.
  \item $-1$ coming from one of the $S^1$ coordinates has as its link
  $S^0 \ast F_{k-1,l}$, a $(2k+l-2)$--disk.
  \item $+1$ coming from one of the $l$ $\{\pm 1\}$ coordinates has
  as its link $\ast_{i=1}^{2k+l-1}S^0$, a $(2k+l-2)$--sphere.
  \item $-1$ coming from one of the $\{\pm 1\}$ coordinates has
  as its link $F_{k,l-1}$, a $(2k+l-2)$--disk.
  \end{enumerate}

  Now we will define a matching on $F_{k,l}$.  First, we order the
  coordinates.  In each $S^1$ coordinate, we also choose one of
  the two copies of $\{\pm 1\}$ to be distinguished.  Associate
  each simplex in $F_{k,l}$ with the simplex obtained by inserting or
  removing $+1$ to or from the first coordinate lacking a $-1$
  (in the distinguished copy of $\{\pm 1\}$ in the case the first
  such coordinate is $S^1$).  Doing this does not change which
  coordinate is the first without a $-1$, so the pairing is
  well-defined.  Every simplex is paired ($\emptyset$ is paired with the vertex
  with a $+1$ in the first coordinate and nothing in any other
  coordinate), so if there are no cycles in this matching, $F_{k,l}$.

  Suppose there were a cycle.  It would have to be of the form:
  \begin{equation*}
  \sigma_0 \lessdot \tau_0 \gtrdot \sigma_1 \lessdot \tau_1 \gtrdot \sigma_2
  \lessdot \ldots \lessdot \tau_{s-1} \gtrdot \sigma_s = \sigma_0
  \end{equation*}
  where each $\sigma_i$ is paired with $\tau_i$.  Also, for $1 \leq i \leq s$,
  $\sigma_i$ must be $\tau_{i-1}$ minus a vertex $v_i$.  Therefore,
  since this is a cycle, there must be a $j$ such that $\tau_j =
  \sigma_j \cup \{v_i\}$.  For this to be possible, $v_i$ must be a
  $+1$.  Thus, all of the simplices in the cycle must have all the
  same $-1$ coordinates, but if that is the case, the vertex to be added in any $\sigma_i \lessdot \tau_i$ pair
  is always the same, and $v_i$ must be the same for every $i$.  This is a
  contradiction.  Therefore, there are no cycles, and we have a Morse matching with
  a single critical simplex.
\end{proof}

In the proof of Proposition \ref{main4} below, we will use the
following lemma, which is essentially an equivariant version of
Theorem 3.1 of \cite{K06-2}.
\begin{lemma}\label{main3}
Let $G$ be any group and $P$ be a finite poset, $h\colon\thinspace P \rightarrow P$ an order-preserving
poset map such that  for any $x \in P$, $h(x) \geq x$ (or $h(x) \leq
x$). Define $Q$ to be the set of fixed points of $h$. Then $\Delta
P$ collapses simplicially to $\Delta Q$.  In the case that h is a G--poset map, P $\searrow_G$ Q.
\end{lemma}
\begin{proof} We prove it for the case that $h(x) \geq x$, the other case being almost identical.  Since
$P$ is finite, we may choose $N$ large enough so that for all $x$
in $P$, we have $h^N(x)\in Q$. Now let $\sigma \in \Delta P$ be a
chain $x_0 < x_1 < \ldots < x_m$. If $\exists \; i: 0 \leq i \leq
m$ such that $x_i \notin Q$, let $k$ be the largest such $i$. Then
we may insert $h^N(x_k)$ into the chain immediately following
$x_k$ because $x_k < h^N(x_k) \leq h^N(x_{k+1}) = x_{k+1}$ if $k <
m$. Associate to $\sigma$ the chain obtained by inserting
$h^N(x_k)$ or by deleting it in the case $h^N(x_k) = x_{k+1}$.
Since it is an element of $Q$ being inserted or deleted, the
selection of $k$ is not affected, and $x_k$ uniquely determines
the other chain in the pair. Therefore, this matching is
well-defined.  Also, this matching is equivariant if $h$ is a
$G$--map.

Suppose there is a cycle
\begin{equation*}
  \sigma_0 \lessdot \tau_0 \gtrdot \sigma_1 \lessdot \tau_1 \gtrdot \sigma_2
  \lessdot \ldots \lessdot \tau_{s-1} \gtrdot \sigma_s = \sigma_0
\end{equation*}
We have for $1 \leq i \leq s$ that $\sigma_i = \tau_{i-1}
\setminus \{y_i\}$ for some $y_i \in P$. Then there must be some
pair $\sigma_j \lessdot \tau_j = \sigma \cup \{y_i\}$, so $y_i \in
Q$ for all $i$.  Thus every simplex in the cycle has all the same
elements of $P \setminus Q$, so $\exists \; x \in P \setminus Q$
that is the greatest such element in every simplex.  Hence $\tau_j
= \sigma_j \cup \{h^N(x)\}$ for all $j$, and $y_i = h^N(x)$ for
all $i$.  This is a contradiction because the same element is
being added and deleted in consecutive steps.  Therefore, we have
a Morse matching whose critical simplices are exactly the elements
of $\Delta Q$, a subcomplex of $\Delta P$.  Thus, $\Delta P$
$G$--collapses to $\Delta Q$.
\end{proof}

\begin{prop}\label{main4}
$\Delta K$ simplicially G--collapses to $\Delta S$.
\end{prop}
\begin{proof} The collapsing will occur in three steps. Define
\begin{equation*}
K_1 := \{\phi = \left(
   \begin{array}{ccc}
    A & B \\
    C & D \\
   \end{array}\right) \in K \quad | \quad A\cap C \neq \emptyset \}
\end{equation*}
\begin{equation*}
K_2 := \{\phi \in K_1 \quad | \quad A = C \}
\end{equation*}
 First, we collapse $\Delta K$ to $\Delta K_1$.  Let $\sigma$ be a chain of the form
\begin{equation*}
\phi_0 < \phi_1 < \ldots < \phi_{m-1} < \phi_m
\end{equation*}
 in $\Delta K$ where $\phi_i = \left(
   \begin{array}{ccc}
    A_i & B_i \\
    C_i & D_i \\
   \end{array}\right)$.  If $A_0 \cap C_0 = \emptyset$, we want to pair $\sigma$ with another chain
for which that is also true.  Find the last $k$ such that $A_k
\cap C_k = \emptyset$. Then compare $B_k$ and $D_k$ to $B_m$ and
$D_m$. If $B_k = B_m$ and $D_k = D_m$, pair $\sigma$ with the
chain obtained by adding to (or deleting from) the end of $\sigma$
the element $\left(
   \begin{array}{ccc}
    A_m\cup (B_m \cup D_m)^c & B_m \\
    C_m\cup (B_m \cup D_m)^c & D_m \\
   \end{array}\right)$.  Otherwise, find the first $l > k$ where $B_l
   \neq B_k$ or $D_l \neq D_k$.  Now pair $\sigma$ with the chain
   obtained by inserting (or removing if it equals $X_{l-1}$) $\left(
   \begin{array}{ccc}
    A_l & B_{l-1} \\
    C_l & D_{l-1} \\
   \end{array}\right)$ before $\phi_l$.
Nowhere are we inserting or deleting elements with $A\cap C =
\emptyset$, so the selection of $k$ is not affected.  In the
second case, we are inserting or deleting an element with $B_{l-1}
= B_k$ and $D_{l-1} = D_k$, so the selection of $l$ is not
affected.  Therefore, the matching is well-defined. The critical
simplices are exactly those where $A_0 \cap C_0 \neq \emptyset$,
forming $\Delta K_1$, a subcomplex. Therefore, if there are no
cycles, we have a collapsing from $\Delta K$ to $\Delta K_1$.
Also, the pairings are chosen equivariantly, so we will have a
$G$--collapse.

Suppose we have a cycle
\begin{equation*}
  \sigma_0 \lessdot \tau_0 \gtrdot \sigma_1 \lessdot \tau_1 \gtrdot \sigma_2
  \lessdot \ldots \lessdot \tau_{s-1} \gtrdot \sigma_s = \sigma_0
\end{equation*}
Again, for $1 \leq i \leq s$, $\sigma_i$ is obtained from $\tau_{i-1}$ by deleting an
element $\psi_i$, so there must be a pair $\sigma_j \lessdot \tau_j =
\sigma_j \cup \{\psi_i\}$ coming from our matching. Therefore, $\psi_i
\in K_1$ for all $i$, which means that all the simplices in our
cycle have all of the same elements with $A \cap C = \emptyset$.
Thus, they all have the same $\phi_k$, so $B_k$ and $D_k$ are fixed
and we know that every $\psi_i$ has them as its second column.  As a
result, the elements after $\phi_k$ that have $B \neq B_k$ or $D \neq
D_k$ are not changing as we move through the cycle, implying that
$\psi_i$ is the same for all $i$.  This is a contradiction, so our
matching has no cycles.  This proves that $\Delta K$ $G$--collapses to
$\Delta K_1$.

The next two collapsings are proved by Lemma \ref{main3}. For the
first, we define $h_1: K_1 \rightarrow K_1$ by $h_1(\phi) =
\left(\begin{array}{ccc}
    A\cap C & B \\
    A\cap C & D \\
   \end{array}\right)$.  This is an order-preserving $G$--poset map, and $h_1(\phi) \leq
   \phi$.  The fixed point set of $h_1$ is exactly $K_2$, so Lemma
   \ref{main3} implies that $\Delta K_1$ $G$--collapses to $\Delta
   K_2$.  For the second collapsing, we now define
   $h_2: K_2 \rightarrow K_2$ by $h_2(\phi) = \left(
   \begin{array}{ccc}
    A & B\cup D \\
    A & B\cup D \\
   \end{array}\right)$.  This is an order-preserving $G$--poset map, $h_2(\phi) \geq
   \phi$, and the fixed point set is $S$.  Therefore, the same lemma
   implies that $\Delta K_2$ $G$--collapses to $\Delta S$.  Hence,
   $\Delta K$ $G$--collapses to $\Delta S$.
\end{proof}

\begin{theorem}\label{main5}
$|\Delta K|$ is a $G$--regular neighborhood of $|\Delta S|$ with
boundary $|\Delta L|$.
\end{theorem}
\begin{proof} $G$ acts freely outside of $|\Delta S|$, so $\partial |\Delta K|$ is $G$--bicollarable in $|\Delta M|$.
Now the theorem follows immediately from \ref{reg16} (the collapsing
criterion for $G$--regular neighborhoods) and Propositions
\ref{main1} and \ref{main4}.
\end{proof}

Now our main result follows easily:

\begin{maintheorem}
The regular cell complex $\Hom_{\{3\}}(P_4,K_n)$ is a PL manifold with boundary
$\Hom_{\{3\}}(C_5,K_n)$ and is equivariantly homeomorphic (with respect to the involution described above) to $N :=
\{(x,y) \in S^{n-2} \times S^{n-2} \; | \; x \cdot y \geq 0 \}$,
where the involution on $N$ interchanges $(x,y)$ with $(y,x)$.
The Stiefel manifold $V_{n-1,2} = \partial N$ is therefore equivariantly homeomorphic to $\Hom_{\{3\}}(C_5,K_n)$, which is equivariantly homeomorphic to $\Hom(C_5,K_n)$.
\end{maintheorem}
\begin{proof}
It follows from \ref{Gcellcx} that we have $\Hom_{\{3\}}(P_4,K^n)
\approx_G |\Delta K|$ with the subcomplex $\Hom_{\{3\}}(C_5,K^n)
\approx_G |\Delta L|$.  Because $|\Delta K|$ and $N$ are both
$G$--regular neighborhoods of the diagonal, they are equivariantly
homeomorphic by \ref{reg4}.
\end{proof}

The Stiefel manifold $V_{n-1,2}$ has a natural action of the
orthogonal group $O_2$ (with the Grassmannian as the quotient).
The equivariant homeomorphism above is with respect to a single
reflection of $O_2$.  The multimorphism complex $\Hom(C_5,K_n)$
does not have a combinatorial $O_2$ action; however, there is the
induced action of the dihedral group $D_5$ (a subgroup of $O_2$)
which is the group of symmetries of the cycle $C_5$.  It seems
natural to ask:
\begin{question}
Is $\Hom(C_5,K_n)$ equivariantly homeomorphic to $V_{n-1,2}$ with
respect to the action of the dihedral group $D_5$?
\end{question} Unfortunately, neither of the smaller models
$\Hom_{\{3\}}(C_5,K_n)$ or $\Hom_{\{2,4\}}(C_5,K_n)$ is
$D_5$--invariant, so it seems that one needs to work with
$\Hom(C_5,K_n)$ which does not have an obvious $D_5$--equivariant
embedding into $S^{n-2} \times S^{n-2}$.  Also, a good
(equivariant) combinatorial candidate for $N$ is missing, which is
the obstacle to applying the methodology above to answer this
question positively.





\bibliographystyle{model1b-num-names}

\begin{thebibliography}{99}

\bibitem{A87}\textbf{N Alon}, \emph{Splitting necklaces,} Adv. in
Math., 63 (1987), 247-253

\bibitem{BSS81}\textbf{I B\'{a}r\'{a}ny, S B Shlossman, and A Sz\H{u}cs}, \emph{On a
topological generalization of a theorem of Tverberg,} J. London
Math. Soc., 23 (1981), 158-164

\bibitem{BK06}\textbf{E Babson and D N Kozlov}, \emph{Complexes of graph
homomorphisms,} Isr. J. Math., 152 (2006), 285-312

\bibitem{BK07}\textbf{E Babson and D N Kozlov}, \emph{Proof of the  Lov\'{a}sz
conjecture,} Annals of Math., 165 (2007), 965-1007

\bibitem{BLSWZ99}\textbf{A Bj\"{o}rner, M Las Vergnas, B Sturmfels, N
White, and G M Ziegler}, \emph{Oriented matroids,} 2nd ed., Encyc.
of Math. and its Appl. 46, Cambridge Univ. Press (1999)

\bibitem{C05}\textbf{P Csorba} \emph{Non-tidy spaces and graph
colorings}, Ph.D. thesis, ETH Z\"{u}rich (2005)

\bibitem{CL06}\textbf{P Csorba and F H Lutz}, \emph{Graph coloring
manifolds,} Contemp. Math., 423 (2006), 51-69

\bibitem{F98}\textbf{R Forman}, \emph{Morse theory for cell complexes,}
Adv. in Math., 134 (1998), 90-145

\bibitem{I00}\textbf{S Illman}, \emph{Existence and uniqueness of equivariant triangulations of smooth proper $G$--manifolds with some applications to equivariant Whitehead torsion,} J. reine angew. Math., 524 (2000), 129-183

\bibitem{K06-1}\textbf{D N Kozlov}, \emph{Cobounding odd cycle
colorings,} Elec. Research Announ. AMS, 12 (2006), 53-55

\bibitem{K06-2}\textbf{D N Kozlov}, \emph{Collapsing along monotone
poset maps,} Int. J. Math. and Math. Sciences, (2006),
Art. ID 79858, 8 pp.

\bibitem{K07}\textbf{D N Kozlov}, \emph{Combinatorial algebraic
topology,} Springer (2007)

\bibitem{K08}\textbf{D N Kozlov}, \emph{Cohomology of colorings of
cycles,} Amer. J. Math., 130 (2008), 829-857

\bibitem{L78}\textbf{L  Lov\'{a}sz}, \emph{Kneser's conjecture, chromatic
number and homotopy,} J. Comb. Theory, Ser. A, 25 (1978),
319-324

\bibitem{LW69}\textbf{A Lundell and S Weingram}, \emph{The topology of CW
complexes,} Van Nostrand Reinhold Company (1969)

\bibitem{O87}\textbf{M \"{O}zayd\i n}, \emph{Equivariant maps for the
symmetric group,} Univ. of Wisconsin preprint (1987), 17pp., Available at: http://digital.library.wisc.edu/1793/63829, Accessed 12/17/2012

\bibitem{RS82}\textbf{C P Rourke and B J Sanderson}, \emph{Introduction to
piecewise-linear topology,} Springer-Verlag (1982)

\bibitem{S08}\textbf{C Schultz}, \emph{Small models of graph colouring
manifolds and the Stiefel manifolds $\Hom(C_5,K_n)$,} J.
Comb. Theory, Ser. A, 115 (2008), 94-104

\bibitem{S09-1}\textbf{C Schultz} \emph{A short proof of
$w_1^n(\Hom(C_{2r+1},K_{n+2})) = 0$ for all $n$ and a graph
colouring theorem by Babson and Kozlov,} Isr. J. Math., 170
(2009), 125-134

\bibitem{S09-2}\textbf{C Schultz} \emph{Graph colorings, spaces of edges and spaces of circuits,}
Adv. in Math., 221 (2009), 1733-1756

\bibitem{Z06}\textbf{G M Ziegler}, \emph{Generalized Kneser coloring
theorems with combinatorial proofs,} Invent. math., 147
(2002), 671-691. Erratum 163 (2006), 227-228

\bibitem{Ziv98}\textbf{R \v{Z}ivaljevi\'{c}}, \emph{User's guide to equivariant methods in combinatorics, I and II,} Publ. Inst. Math. (Beograd) (N.S.), (I) 59(73) (1996), 114-130 and (II) 64(78) (1998), 107-132
\end{thebibliography}







\end{document}